\documentclass{amsart}

\usepackage{amsthm}
\usepackage[leqno]{amsmath}
\usepackage[pagebackref]{hyperref}
\usepackage{amsfonts,amsmath,amssymb,amsthm,amscd,amsxtra}
\usepackage{enumerate,verbatim}
\usepackage[usenames,dvipsnames]{pstricks}
\usepackage[mathscr]{eucal}
\usepackage{amsfonts,amsmath,amssymb,amsthm,amscd,amsxtra}
\usepackage{amssymb,amstext,amsmath,amscd,amsthm,amsfonts,enumerate,graphicx,latexsym}

\usepackage{enumerate,verbatim}
\usepackage[all,2cell,ps]{xy}
\usepackage[notcite, notref]{}
\usepackage[pagebackref]{hyperref}
\usepackage{todonotes}

\usepackage[all,2cell,ps]{xy}
\usepackage[notcite,notref]{}
\usepackage[pagebackref]{hyperref}

\theoremstyle{plain} %default

\newtheorem{thm}{Theorem}[section]

\newtheorem{prop}[thm]{Proposition}
\newtheorem{cor}[thm]{Corollary}

\theoremstyle{definition}

\newtheorem{obs}[thm]{Observation}

\newtheorem{eg}[thm]{Example}

\newtheorem{ques}[thm]{Question}

\newtheorem{rmk}[thm]{Remark}

\numberwithin{equation}{section}
\newcommand{\fm}{\mathfrak{m}}

\newcommand{\fp}{\mathfrak{p}}

\newcommand{\CC}{\mathbb{C}}

\newcommand{\ZZ}{\mathbb{Z}}
\newcommand{\NN}{\mathbb{N}}

\newtheorem{chunk}[thm]{\hspace*{-1.065ex}\bf}
\newcommand{\up}[1]{{{}^{#1}\!}}

 \DeclareMathOperator{\Ass}{Ass}

\def\S{\operatorname{\mathsf{S}}}

\def\H{\operatorname{\mathsf{H-dim}}}
\def\HH{\operatorname{\mathsf{H}}}

\def\CI{\operatorname{\mathsf{CI-dim}}}
\def\G-dim{\operatorname{\mathsf{G-dim}}}

\def\cx{\operatorname{\mathsf{cx}}}
\def\pd{\operatorname{\mathsf{pd}}}
\def\id{\operatorname{\mathsf{id}}}
\def\Gid{\operatorname{\mathsf{Gid}}}

\def\cpd{\operatorname{C-\mathsf{pd}}}

\def\cid{\operatorname{C-\mathsf{id}}}

\def\syz{\operatorname{\mathsf{syz}}}

\def\Gcid{\operatorname{\mathsf{C-Gid}}}
\def\CMid{\operatorname{\mathsf{CMid}}}
\def\Gxcid{\operatorname{\mathsf{\overline{C}-Gid}}}

\def\Tr{\mathsf{Tr}\hspace{0.01in}}

\def\depth{\operatorname{\mathsf{depth}}}

\def\Ext{\operatorname{\mathsf{Ext}}}

\def\Hom{\operatorname{\mathsf{Hom}}}

 \DeclareMathOperator{\gr}{grade}

\DeclareMathOperator{\Supp}{Supp}
\DeclareMathOperator{\Spec}{Spec}

\def\Tor{\operatorname{\mathsf{Tor}}}

\newcommand{\tensor}{\otimes^{\bf L}}
\DeclareMathOperator{\Rhom}{\textbf{R}Hom}

\def\urltilda{\kern -.15em\lower .7ex\hbox{\~{}}\kern .04em}
\def\urldot{\kern -.10em.\kern -.10em}\def\urlhttp{http\kern -.10em\lower -.1ex
\hbox{:}\kern -.12em\lower 0ex\hbox{/}\kern -.18em\lower 0ex\hbox{/}}

\begin{document}
%\baselineskip=15pt
%\listoftodos{}

\title[Rigid modules]
{Homological dimensions of rigid modules} %REGULARITY via Rigidity}
%Regularity of local rings and the
%The idiosyncrasies of rigid-test modules

\author[Celikbas, Gheibi, Rahro Zargar, Sadeghi]
{Olgur Celikbas, Mohsen Gheibi, \\  Majid Rahro Zargar and Arash Sadeghi }
\address{Olgur Celikbas, Department of Mathematics, University of Connecticut, Storrs, CT 06269, USA}
\email{olgur.celikbas@uconn.edu}

\address{Mohsen Gheibi, Department of Mathematics, University of Nebraska–-Lincoln, Lincoln, NE 68588, USA}
\email{mohsen.gheibi@gmail.com}

\address{Majid Rahro Zargar, School of Mathematics, Institute for Research in Fundamental Sciences (IPM), P.O. Box: 19395-5746, Tehran, Iran.}
\email{zargar9077@gmail.com}

\address{Arash Sadeghi, School of Mathematics, Institute for Research in Fundamental Sciences (IPM), P.O. Box: 19395-5746, Tehran, Iran.}
\email{sadeghiarash61@gmail.com}


\date{\today}
\maketitle{}

\begin{abstract}
We obtain various characterizations of commutative Noetherian local rings $(R, \fm)$ in terms of homological dimensions of certain finitely generated modules. %Our argument has a series of consequences in different directions. 
For example, we establish that $R$ is Gorenstein if the Gorenstein injective dimension of the maximal ideal $\fm$ of $R$ is finite. Furthermore we prove that $R$ must be regular if a single $\Ext_{R}^{n}(I,J)$ vanishes for some integrally closed $\fm$-primary ideals $I$ and $J$ of $R$ and for some integer $n\geq \dim(R)$.  Along the way we observe that local rings that admit maximal Cohen-Macaulay Tor-rigid modules are Cohen-Macaulay.
\end{abstract}

\setcounter{tocdepth}{1}
\tableofcontents

\section{Introduction} Throughout $R$ is a commutative Noetherian local ring with unique maximal ideal $\fm$ and residue field $k$, and all modules over $R$ are assumed to be finitely generated.

It is well-known that the projective dimension of an $R$-module $M$ is determined by the vanishing of $\Ext_{R}^{n}(M,k)$, i.e., if $\Ext_{R}^{n}(M,k)=0$ for some positive integer $n$, then $\pd(M)\leq n-1$. In fact $\pd(M)=\sup\{i\in \ZZ: \Ext^i_R(M,k)\neq 0\}$. Furthermore
it follows from classical theorems of Auslander, Buchsbaum and Serre that finiteness of the projective or the injective dimension of the residue field $k$ characterizes the ring itself: $R$ is regular if $\pd(k)<\infty$ or $\id(k)<\infty$; see \cite[2.2.7 and 3.1.26]{BH}.

The main task in this paper is to introduce a class of modules, called \emph{rigid-test} modules, that replace the residue field $k$ in the aforementioned classical results; see (\ref{testrigid}) for the definition. A special case of our main result, Theorem \ref{lemImp}, can be summarized as follows; see also Corollaries \ref{corofthm-1} and \ref{cpcidcor}.

\begin{thm} \label{thmint} Let $(R, \fm)$ be a local ring and let $M$ and $N$ be a nonzero $R$-modules. Assume $N$ is a rigid-test module (e.g., $N=k$).
\begin{enumerate}[\rm(i)]
\item If $\Ext^n_{R}(M,N)=0$ for some $n\geq \depth(N)$, then $\pd(M)\leq n-1$.
\item $\pd(M)=\sup\{i\in \ZZ:  \Ext^i_{R}(M,N)\neq 0\}$.
\item If $\pd(N)<\infty$ or $\id(N)<\infty$, then $R$ is regular.
\end{enumerate}
\end{thm}

To motivate our approach, let us remark Corso, Huneke, Katz and Vascencelos \cite[3.3]{CHKV} established that integrally closed $\fm$-primary ideals are rigid-test modules; see (\ref{inttest}). Thus, for such ideals $I$ and $J$, if $\Ext_{R}^{n}(I,J)=0$ for some $n\geq \dim(R)$, then it follows from Theorem \ref{thmint} that $\pd(I)<\infty$, and hence $R$ is regular;  see (\ref{testrigid}). Furthermore, if $R$ is a two-dimensional complete normal local rational singularity with an algebraically closed residue field, by a result of Lipman \cite[7.1]{Lipman}, $I^i$ is integrally closed, and hence rigid-test, for all $i\geq 1$. Consequently Theorem \ref{thmint} yields the following result; see also (\ref{dis3}).

\begin{cor} \label{corthmint} Let $(R, \fm)$ be a two-dimensional complete normal local domain with an algebraically closed residue field. Assume that $R$ has a rational singularity. If $\Ext^n_{R}(I^r, J^s)=0$ for some integrally closed $\fm$-primary ideals $I$ and $J$ of $R$, and for some positive integers $n$, $r$, $s$, then $R$ is regular.
\end{cor}

Examples of two-dimensional rational singularities include hypersurface rings (rational double points), such as $R=\CC[\![x,y,z]\!]/(x^2+y^2+z^2)$.
%Corollary \ref{corthmint} corroborates a vanishing result obtained in \cite{CeD}; see (\ref{dis3}).

Our argument has applications in several directions. We use Theorem \ref{thmint} and deduce the following characterization of regularity from a beautiful result of Avramov, Hochster, Iyengar and Yao \cite[1.1]{AHIY}; see (\ref{AHiy}) and Corollary \ref{CorintrooGen}.

\begin{cor} \label{Corintroo} Let $R$ be a complete local ring of prime characteristic $p$ with a perfect residue field, and let $M$ and $N$ be nonzero $R$-modules with $\Ext^i_{R}(\up{\varphi^n}M, N)=0$ for some $i\geq \depth(N)$ and $n\geq 1$. If $N$ is a rigid-test module, then $R$ is regular.
\end{cor}

Here $\up{\varphi^n}M$ is the $R$-module $M$ with the $R$-action given by the nth iterate of the Frobenius endomorhism $\varphi$; see (\ref{rmk0}). As an example, we note that nonzero modules of infinite projective dimension are rigid-test modules over $R=\mathbb{F}_{p}[\![x,y,z]\!]/(xy-z^2)$, with $p$ is an odd prime; thus Corollary \ref{Corintroo} implies that $\Ext^{n+1}_{R}(\up{\varphi^e}M, \up{\varphi^r}N)\neq 0$ for all positive integers $e, n, r$, and for all nonzero $R$-modules $M$, $N$; see (\ref{onGolod}) and (\ref{Long's result}).

Theorem \ref{thmint} determines the Gorensteinness of $R$ via the \emph{Gorenstein injective dimension}, a refinement of the usual injective dimension introduced by Enochs and Jenda \cite{EJ}. We prove in  Corollary \ref{corsonuc1} that $R$ is Gorenstein if the Gorenstein injective dimension $\Gid(\fm)$ of the maximal ideal $\fm$ of $R$ is finite. This, combined with the results in the literature, seems to give a fairly complete picture: $R$ is Gorenstein if and only if at least one of the dimensions $\Gid(\fm)$, $\Gid(R)$ or $\Gid(k)$ is finite; see also (\ref{Gidremark}) and Avramov's remark following Question \ref{soru}.

A rigid-test module is, by definition, Tor-rigid \cite{Au} and a test module (for projectivity) in the sense of \cite{CDtest}; see (\ref{testrigid}). Among those already discussed, there are quite a few motivations to study test and rigid-test modules: it was established in \cite[3.7]{CDtest} that if the dualizing module of a Cohen-Macaulay ring is a test module, then there are no nonfree totally reflexive modules. Proposition \ref{corgen} extends \cite[3.7]{CDtest} and establishes that there are no nonfree totally reflexive modules if there exists a nonzero test module over $R$ -- not necessarily maximal Cohen-Macaulay -- of finite injective dimension. Another motivation for us to introduce rigid-test modules comes from the fact that the hypothesis -- $N$ is a test module -- in Theorem \ref{thmint} cannot be dropped in general. Tor-rigidity has remarkable consequences \cite{Au, Da1}, but if $N$ is a Tor-rigid module, which is not a test module, i.e., not a rigid-test module, then the vanishing of $\Ext^n_{R}(M,N)$, even for all $n\gg 0$, does not necessarily force $M$ to have finite projective dimension in general; see Theorem \ref{thmint} and Example \ref{eg22}.

Cohen-Macaulay local rings admit maximal Cohen-Macaulay Tor-rigid modules, e.g., a high syzygy of the residue field is such an example. On the other hand, examples of non Cohen-Macaulay rings that admit maximal Cohen-Macaulay modules are easy to find; see (\ref{MCMexamples}). In proving our main result, we discover the following, which came as a surprise to us; see (\ref{rigidcons}).

\begin{obs} If a local ring $R$ admits a (finitely generated) maximal Cohen-Macaulay  Tor-rigid module, then $R$ is Cohen-Macaulay. %$\phantom{}$
\end{obs}

We make various observations  in sections 3 and 4, and prove our main result, Theorem \ref{lemImp}, in section 5. Sections 6 and 7 are devoted to applications of our argument. We also collect examples of test and rigid-test modules from the literature in Appendix A; Appendix B colloborates a result of Iyama and Wemyss \cite{IW}.

\section{Definitions}

\begin{chunk} (\cite{Au}) \label{Torrigid} An $R$-module $M$ is said to be \emph{Tor-rigid} provided that the following holds for all $R$-modules $N$:
\begin{equation}\notag{}
\text{If } \Tor_{n}^{R}(M,N)=0 \text{ for some }  n\geq 1, \text{then } \Tor_{n+1}^{R}(M,N)=0.
\end{equation}
\end{chunk}
The notion of Tor-rigidity were initially used in the study of the Koszul complex; it was later formulated and analyzed for modules by Auslander; see \cite{Au}. An interesting result of Lichtenbaum \cite[Theorem 3]{Li} shows that modules over regular local rings, and those of finite projective dimension over hypersurfaces -- quotient of power series rings over fields -- are Tor-rigid.

\begin{chunk} (\cite[1.1]{CDtest}) \label{Test} An $R$-module $M$ is said to be a \emph{test module for projectivity} provided that the following holds for all $R$-modules $N$:
\begin{equation}\notag{}
\text{If } \pd(N)=\infty \text{, then} \Tor_{n}^{R}(M,N)\neq 0 \text{ for infinitely many integers } n.
\end{equation}
We will call a test module for projectivity simply a \emph{test module}.
\end{chunk}

Motivated by a question of Lichtenbaum \cite[page 226, question 4]{Li}, we define:

\begin{chunk} \label{testrigid}
A Tor-rigid test module is called a \emph{rigid-test module}. More precisely, $M$ is called a rigid-test module provided that the following holds for all $R$-modules $N$:
\begin{equation}\notag{}
\text{If }\Tor_{n}^{R}(M,N)=0 \text{ for some } n\geq 1 \text{, then } \Tor_{n+1}^{R}(M,N)=0 \text{ and} \pd(N)<\infty.
\end{equation}
\end{chunk}

Dao, Li and Miller \cite{DLC} defined \emph{strong-rigidity} to study Tor-rigidity of the Frobenius endomorphism over Gorenstein rings:

\begin{chunk} (\cite[2.1]{DLC}) \label{stronglyrigid}
An $R$-module $M$ is said to be \emph{strongly-rigid} provided that the following  holds for all $R$-modules $N$:
\begin{equation}\notag{}
\text{If } \Tor_{n}^{R}(M,N)=0 \text{ for some } n\geq 1 \text{, then } \pd(N)<\infty.
\end{equation}
\end{chunk}

It follows from the definition that rigid-test modules are strongly-rigid, but we do not know whether the converse is true in general. Most of our results work for strongly-rigid modules. However, to obtain the conclusion of Theorem \ref{thmint}(i) when $n\geq \depth(N)$, we need Tor-rigidity; see Theorem \ref{lemImp} and Corollary \ref{corofthm-1}. This leads us to pose the following question for further study.

\begin{ques} \label{qintro} Let $R$ be a local ring and let $M$ be an $R$-module. If $M$ is strongly-rigid, then must $M$ be a rigid-test module, or equivalently, must $M$ be Tor-rigid?
\end{ques}

We give some relations between the above definitions in a diagram form:

\vspace*{-9ex}

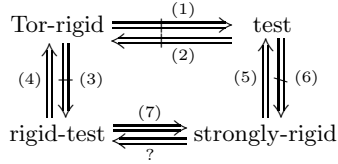
$$
{\small
\xymatrix@C=3em@R=3em{
& \\
\text{Tor-rigid} \ar@{=>}[r]|{\object@{|}}^-{\;\;\;\; (1)}  \ar@{=>}[d] <0.8ex>|{\object@{|}}^-{ \;(3)}  & \text{\;\;test} \ar@{=>}[l]<1.5ex>|{\object@{|}}^-{\;\;\;\; (2)}    \ar@{=>}[d]<1.5ex>|{\object@{/}}|{}^-{\; (6)}  &  \\ \text{rigid-test} \ar@{=>}[u]<0.8ex>|{\object@{}}^-{\;\;\;\; (4)} \ar@{=>}[r]<0.6ex>|{\object@{}}^-{(7)}
& \text{strongly-rigid} \ar@{=>}[l]<0.7ex>|{\object@{}}^-{?}  \ar@{=>}[u] |{\object@{}}|{}^-{\;\;\;\; (5)} }} \\
$$
\vspace{0.1in}
%\vspace*{-2ex}

The implications in the diagram can be justified as follows:%$\phantom{.}$

\noindent (1) and (3): see Example \ref{eg22}.\\
\noindent (2) and (6): see Example \ref{eg33}.\\
% Let $R=\CC[\![x,y]\!]/(xy)$, $M=R/(x+y)$, $T=R/(x)$ and $N=R/(y)$. Then $\pd(M)=1$ so that $M$ is Tor-rigid, but it is not a test module; see (\ref{Torrigid}) and (\ref{Test}). Moreover $\Tor^R_1(T,N)=0$ with $\pd(T)=\infty=\pd(N)$. Thus $T$ is not strongly-rigid, but it is a test module; see (\ref{onGolod}).\\
\noindent (4), (5) and (7): these follow from the definitions; see (\ref{Torrigid}), (\ref{testrigid}) and (\ref{stronglyrigid}).

\section{Projective and injective dimensions via rigid modules} Let $R$ be a local ring. If $N$ is a nonzero rigid-test module over $R$, then the vanishing of $\Tor_{i}^R(N,N)$ is not mysterious at all: it follows from the definition -- unless $R$ is regular -- that $\Tor_{i}^R(N,N)\neq 0$ for all $i\geq 0$; see (\ref{testrigid}). Hence it seems interesting to consider the vanishing of $\Tor_{i}^{R}(M,N)$ when $N$ is a rigid-test module and $M$ is an arbitrary $R$-module. In particular we seek to find whether the vanishing of $\Tor_{i}^{R}(M,N)$ for all $i\gg 0$ yields the exact value of the projective dimension of $M$. Auslander remarked that, if $\depth(N)=0$ and $\pd(M)=s<\infty$, then $\Tor_{s}^{R}(M,N)\neq 0$; see \cite[1.1]{Au}. Therefore an immediate observation is:

\begin{chunk} \label{Auobs} If $R$ is a local ring and $N$ is a test module such that $\depth(N)=0$, then it follows that $\pd(M)=\sup\{i\in \ZZ: \Tor_i^R(M,N)\neq 0\}$; see (\ref{Test}).
\end{chunk}

A rigid-test module of positive depth does not necessarily detect the exact value of the projective dimension via the vanishing of Tor in general; cf. Theorem \ref{thmint}.

\begin{eg} \label{egg1} Let $R=k[\![x,y,z]\!]/(xy)$, $T=R/(x)$ and let $N=T\oplus  \Omega T=R/(x)\oplus R/(y)$. Then $\depth(N)=2$ and $N$ is Tor-rigid; see \cite[1.9]{Mu}. Moreover, since $\pd(N)=\infty$, it follows that $N$ is a rigid-test module; see (\ref{onGolod}). Setting $M=R/(z)$, we see that $1=\pd(M) \neq \sup \{ i \in \ZZ:  \Tor_{i}^{R}(M,N) \neq 0\}=0$.
\end{eg}

If $N$ is a rigid-test module that is not necessarily of depth zero, Proposition \ref{th3} can be useful to detect the projective dimension of $M$; see also Remark \ref{Jorrmrk}. In the following
$\syz(N)$ denotes the largest integer $n$ for which $N$ can be an $n$th syzygy module in a
minimal free resolution of an $R$-module. The assumption that $\depth(N)=\syz(N)$ in Proposition \ref{th3} holds, for example, when $N$ is reflexive and $N_{\fp}$ is free for all prime ideals $\fp$ of $R$ with  $\fp \neq \fm$; see \cite[3.9]{EG}.

\begin{prop}\label{th3} Let $R$ be a local ring and let $M$ and $N$ be nonzero $R$-modules. Assume $N$ is a rigid-test module and that $\syz(N)=\depth(N)\leq\pd(M)$. Then
$$\sup\{i\in \ZZ: \Tor_i^R(M,N)\neq0\}=\pd(M)-\depth(N).$$
\end{prop}

\begin{proof}
We may assume $\pd(M)<\infty$; see (\ref{testrigid}). Set
$\pd(M)=n$, $\depth(N)=t$, and $q=\sup\{i\in \ZZ:
\Tor_i^R(M,N)\neq0\}$. We proceed by induction on $t$ and prove that $\Tor_i^R(M,N)\neq0$ for
all $i=0, \dots, n-t$. Notice, since $N$ is Tor-rigid, it is enough
to show that $\Tor_{n-t}^R(M,N)\neq 0$.

If $t=0$, then it follows from Auslander's remark that $\Tor_{n}^R(M,N)\neq 0$; see (\ref{Auobs}). Hence assume $t\geq 1$, and pick a non zero-divisor $x$ on $N$. Then one can see that there is a long exact sequence of the form:
\begin{equation}\notag{}
\cdots\rightarrow\Tor_{n-t+1}^R(M,N)\rightarrow\Tor_{n-t+1}^R(M,N/xN)\rightarrow\Tor_{n-t}^R(M,N)\rightarrow\cdots
\end{equation}
It is easy to see that $N/xN$ is a rigid-test over $R$; see \cite[2.2]{CDtest}.
Thus the induction hypothesis yields $\Tor_{n-t+1}^R(M,N/xN)\neq0$. Therefore
$\Tor_{n-t}^R(M,N)\neq0$. In particular we have that $q\geq n-t$.

Let $\fp\in\Ass(\Tor_q^R(M,N))$. Then the
depth formula \cite[ 1.2]{Au} implies that:
\begin{equation}\tag{\ref{th3}.1}
\pd(M_\fp)-\depth_{R_\fp}(N_\fp)=\depth(R_{\fp})-\depth_{R_{\fp}}(M_{\fp})-\depth_{R_\fp}(N_\fp)=q.
\end{equation}
If $\depth_{R_\fp}(N_\fp)\geq\depth(N)$, then it follows that:
$$q=\pd(M_\fp)-\depth_{R_\fp}(N_\fp)\leq\pd(M)-\depth(N)=n-t.$$
This shows $q=n-t$, and hence completes the proof.

Next suppose $\depth_{R_\fp}(N_\fp)<\depth(N)=t$. Since $\depth(N)=\syz(N)$, we know that $N$ is a $t$-th syzygy module. This implies that $\depth_{R_\fp}(N_\fp)\geq\min\{t,\depth R_\fp\}$; see \cite[1.3.7]{BH}. Hence
$\depth_{R_\fp}(N_\fp)\geq\depth R_\fp$. Therefore, by (\ref{th3}.1), we deduce that
$q=0$. Now the fact $\Tor_{n-t}^R(M,N)\neq0$ yields that $n-t=0$, i.e., $q=n-t$.
\end{proof}

\begin{rmk} \label{Jorrmrk}  Jorgensen \cite[2.2]{JAB} proved, if $M$ is a module over a local ring $R$ with $\pd(M)<\infty$, then $q^R(M,N)=\sup\{i\in \ZZ: \Tor_i^R(M,N)\neq0\}$ is equal to $\sup\{\pd_{R_{\fp}} (M_{\fp})-\depth_{R_{\fp}} (N_{\fp}):\fp\in \Supp(M\otimes_{R}N)  \}$.
Therefore, if $\pd(M)<\infty$, one can deduce from Jorgensen's result that
$q^R(M,N) \geq \pd(M)-\depth(N)$.

In case $\syz(N)=\depth(N)\leq\pd(M)<\infty$, Proposition \ref{th3} establishes the equality $q^R(M,N) = \pd(M)-\depth(N)$ without appealing to \cite[2.2]{JAB}. Notice, by Example \ref{egg1}, the hypothesis $\depth(N)\leq\pd(M)$ is required, but we do not know whether  the condition $\syz(N)=\depth(N)$ is essential.
\end{rmk}

%\section{Detecting the injective dimension via strongly-rigid modules.}
If $(R, \fm, k)$ is a local ring and $N$ is a nonzero $R$-module, then the $k$-vector spaces $\Ext_{R}^n(k,N)$ are nonzero for all $n$, where $\depth(N)\leq n \leq \id(N)$; see \cite[Theorem 2]{R1}.
%This was proved by Fossum, Foxby, Griffith and Reiten \cite{Fossum} in the equicharacteristic case, and by Roberts \cite{Robertss} in general.
In other words, if $\Ext_R^n(k,N)=0$ for some $n\geq \depth(N)$, then $\id(N)<\infty$. Since $k$ is strongly-rigid, this leads us to pose the following question; see also (\ref{stronglyrigid}).

\begin{ques} \label{injque} Let $R$ be a local ring and let $M$ and $N$ be a nonzero $R$-modules. Assume $M$ is strongly-rigid and $\Ext_{R}^n(M,N)=0$ for some $n\geq \depth(N)$. Then must $N$ have finite injective dimension?
\end{ques}

In case $M$ is a test module (not necessarily strongly-rigid) and $R$ has a dualizing complex (i.e., $R$ is a homomorphic image of a Gorenstein ring), it follows from \cite[3.2]{CDtest} that $\id(N)<\infty$ if and only if $\Ext_{R}^i(M,N)$ vanishes for all $i\gg 0$. Here our aim is to examine the case where $M$ is strongly-rigid, and a single $\Ext_{R}^n(M,N)$ vanishes for some $n\geq \depth(N)$. In Proposition \ref{propid} we obtain a partial affirmative answer to Question \ref{injque} over Cohen-Macaulay rings. This, in particular, gives an affirmative answer when $R$ is Artinian; see Corollary \ref{corArt}.

\begin{prop} \label{propid} Let $(R, \fm)$ be a Cohen-Macaulay local ring with a dualizing module and let $M$ and $N$ be nonzero $R$-modules. Assume the following holds:
\begin{enumerate}[\rm(i)]
\item $\pd_{R_{\fp}}(M_{p})<\infty$ for all $p\in \Spec(R)-\{\fm\}$ (e.g., $R$ has an isolated singularity.)
\item $\Ext_R^j(M,N)=0$ for some $j\geq \dim(R)+1$.
\item $M$ is strongly-rigid.
\end{enumerate}
Then $\id(N)<\infty$.
\end{prop}

\begin{proof} As $R$ has a dualizing module, we can consider a maximal Cohen-Macaualy approximation of $N$, i.e., a short exact sequence of $R$-modules
\begin{equation} \tag{\ref{propid}.1}
0 \to Y \to C\to N \to 0,
\end{equation}
where $C$ is maximal Cohen-Macaulay and $\id(Y)<\infty$. Set $n=j-d+\depth(M)$ where $d=\dim(R)$. Applying $\Hom_{R}(M,-)$ to (\ref{propid}.1), we get the following long exact sequence:
\begin{equation} \tag{\ref{propid}.2}
\cdots \to \Ext_R^j(M,Y) \to \Ext_R^j(M,C)\to \Ext_R^j(M,N) \to \cdots
\end{equation}
Note that $\Ext_R^j(M,Y)=0$. So it follows from (\ref{propid}.2) and (ii) that $\Ext_R^j(M,C)=0$. Observe, by (\ref{propid}.1), that $\id(C)<\infty$ if and only if $\id(N)<\infty$. Therefore we may assume $N$ is maximal Cohen-Macaulay. Consider the following standard spectral sequence:
%; see \cite[Theorem 10.62]{Roit}.
$$
E_2^{p,q}=\Ext_{R}^p(\Tor^{R}_q(N^{\dag},M),\omega)\quad\Longrightarrow\quad H^{p+q}=\Ext_{R}^{p+q}(M, N)
$$
Here $N^{\dag}=\Hom(N,\omega)$ and $\omega$ is the dualizing module of $R$. Observe  $\Tor^{R}_q(M,N^{\dag})$ has finite length for all $q\geq 1$: this follows from (i) and the fact that $N^{\dag}$ is maximal Cohen-Macaulay; see \cite[2.2]{Yoshida}. Therefore $E_2^{p,q}=0$ if $q\geq 1$ and $p\ne d$. Furthermore:
$$
\Ext_{R}^{j}(M,N)=H^{j}\cong E_2^{d,n}=\Ext_{R}^d(\Tor^{R}_n(M,N^{\dag}),\omega).
$$
Now, by (ii), the local duality theorem \cite[3.5.11(b)]{BH} yields that $\Tor^{R}_n(M,N^{\dag})=0$. Thus (iii) gives the required conclusion; see also (\ref{stronglyrigid}).
\end{proof}

We will use the following observation several times; see Corollaries \ref{corofthm-1} and \ref{corofthm0}.

\begin{chunk} \label{sifir} Let $R$ be a local ring and let $M$ and $N$ be nonzero $R$-modules. Assume $\Ext^n_{R}(M,N)=0$ for some $n\geq \depth(N)$. If $n=0$, then $\depth(N)=0$ and hence $\Hom(M,N)\neq 0$; see, for example, \cite[1.2.3]{BH}. Therefore $n$ is positive.
\end{chunk}

\begin{cor} \label{corArt} Let $R$ be an Artinian ring and let $M$ and $N$ be nonzero $R$-modules. If $M$ is strongly-rigid and $\Ext_R^n(M,N)=0$ for some $n\geq 0$, then $N$ is injective.
\end{cor}

\begin{proof} In view of (\ref{sifir}), the required result follows from Proposition \ref{propid}.
\end{proof}

If $R$ is an Artinian hypersurface, that is quotient of a power series ring over a field, $M$ is an $R$-module of infinite projective dimension and $\Ext_R^n(M,N)=0$ for some $n\geq 0$, then \cite[5.12]{AvBu} and \cite[4.7]{CeD} show that $N$ is injective. One can recover this result from Corollary \ref{corArt} since each module of infinite projective dimension is strongly-rigid over such an Artinian hypersurface; see (\ref{onGolod}) and (\ref{Long's result}).

\begin{cor} \label{corsifir} Let $R$ be a $d$-dimensional excellent Cohen-Macaulay local ring, $N$ a nonzero $R$-module, and let $I$ be an integrally closed $\fm$-primary ideal of $R$.
Assume $\Ext_{R}^{n}(I,N)=0$ for some $n\geq d$. Then $\id(N)<\infty$.
\end{cor}

\begin{proof} It follows that $I\otimes_{R}\widehat{R}$ is an integrally closed $\fm\widehat{R}$-primary ideal of $\widehat{R}$, where $\widehat{R}$ is the $\fm$-adic completion of $R$; see, for example \cite[19.2.5]{HI}. So we may assume $R$ is complete with a dualizing module. Notice, by (\ref{sifir}), we have that $n\geq 1$. Moreover $I$ is strongly-rigid, and $I_{\fp}$ is free for all $\fp\in \Spec(R)-\{\fm \}$; see (\ref{inttest}). Thus if $\Ext_{R}^{n}(I,N)=0$ for some $n\geq d$, then Proposition \ref{propid} implies that $\id(N)<\infty$.
\end{proof}

\begin{cor} Let $R$ be a $d$-dimensional complete Cohen-Macaulay local ring. If $\Ext_{R}^{n}(I,R)=0$ for some $n\geq d$ and some integrally closed $\fm$-primary ideal $I$ of $R$, then $R$ is Gorenstein.
\end{cor}

It was posed in \cite[3.5-3.6]{CDtest} that whether or not test property is preserved under completion; see (\ref{Test}). An affirmative answer has been recently obtained in \cite{CelWag}:

\begin{chunk} (\cite{CelWag}) \label{testpreserved} Let $R$ be a local ring and let $M$ be a nonzero $R$-module. Then $M$ is a test module over $R$ if and only if $M\otimes_{R}\widehat{R}$ is a test module over the $\fm$-adic completion $\widehat{R}$ of $R$.
\end{chunk}

In light of (\ref{testpreserved}), the excellent hypothesis on $R$ in Corollary \ref{corsifir} can be removed provided there is an affirmative answer to the following longstanding open problem:

\begin{ques} Let $R$ be a local ring. If $M$ is a Tor-rigid module over $R$, then must $M\otimes_{R}\widehat{R}$ be Tor-rigid over $\widehat{R}$?
\end{ques}

\begin{rmk} It is established in Corollary \ref{corsifir} that, if $\Ext_{R}^{n+1}(R/I,N)=0$, then $\id(N)<\infty$. Since $R/I$ has finite length, it is worth noting that the vanishing of $\Ext_{R}^{n}(M,N)$ for an arbitrary $R$-module $M$ of finite length does not force $N$ to have finite injective dimension in general. For example, if $R=k[[x,y]]/(xy)$, $M=R/(x+y)$ and $N=k$, then $\Ext_R^{i}(M,N)=0$ for all $i\geq 2$, but $\id(N)=\infty$.
\end{rmk}

\section{Auslander's transpose and remarks on Tor-rigidity}

\begin{chunk} \label{AuBrsequence} (\cite{AuBr})
Let $M$ be an $R$-module with a projective presentation $P_1\overset{f}{\rightarrow}P_0\rightarrow M\rightarrow 0$. Then the \emph{transpose} $\Tr M$ of $M$
is the cokernel of $f^{\ast}=\Hom_{R}(f,R)$ and hence is given by the exact sequence:
$0\rightarrow M^*\rightarrow P_0^*\rightarrow P_1^*\rightarrow \Tr M\rightarrow 0$.

If $n$ is a positive integer, $\mathcal{T}_n M$ denotes the transpose of the $(n-1)st$ syzygy of $M$, i.e., $\mathcal{T}_n M=\Tr\Omega^{n-1}M$.

There are exact sequence of functors \cite[2.8]{AuBr}:
\begin{equation}\tag{i}
\Tor_2^R(\mathcal{T}_{n+1}M,-)\rightarrow(\Ext^n_R(M,R)\otimes_R-)\rightarrow\Ext^n_R(M,-)\twoheadrightarrow
\Tor_1^R(\mathcal{T}_{n+1}M,-),
\end{equation}
\begin{equation}\tag{ii}
\Ext^1_R(\mathcal{T}_{n+1}M,-)\hookrightarrow\Tor_n^R(M,-)\rightarrow\Hom(\Ext^n_R(M,R),-)
\rightarrow\Ext^2_R(\mathcal{T}_{n+1}M,-)
\end{equation}
\end{chunk}

The next useful fact has been initially addressed by Auslander in his 1962 ICM talk \cite{Au2}.  Auslander's original remark is for regular local rings but,  in case Tor-rigidity holds, it also holds over arbitrary local rings; see (\ref{Torrigid}).

\begin{chunk} (Auslander \cite[Corollary 6]{Au2}; see also Jothilingam \cite{Jot}) \label{Ausobs} Let $M$ and $N$ be nonzero $R$-modules. Assume $N$ is Tor-rigid. Assume further that $\Ext^n_R(M,N)=0$ for some nonnegative integer $n$.

It follows from (\ref{AuBrsequence})(i) that $\Tor_1^R(\mathcal{T}_{n+1}M,N)=0$. This implies, since $N$ is Tor-rigid, that $\Tor_i^R(\mathcal{T}_{n+1}M,N)=0$ for all $i\geq 1$. We can now use (\ref{AuBrsequence})(i) once more and conclude that $\Ext^n_R(M,R)\otimes_R
N=0$. Therefore $\Ext^n_R(M,R)=0$.

In particular, if $M=N$, then $\Tor^{R}_1(\Tr\Omega^nM,\Omega^nM)=0$ so that \cite[3.9]{Yo} implies $\Omega^nM$ is free, i.e., $\pd(M)\leq n-1$.
\end{chunk}

As $\up{\varphi}R$ is Tor-rigid over complete intersection rings, we deduce:

\begin{chunk} \label{rmkF} %$\phantom{}$
%\begin{enumerate}[\rm(i)]
%\item A semidualizing module is Tor-rigid if and only if it is free of rank one.
%\item
Assume $R$ is an F-finite local complete intersection ring with prime characteristic $p$. If $\Ext^i_{R}(\up{\varphi^n}R, \up{\varphi^n}R)=0$ for some positive integers $i$ and $n$, then it follows from (\ref{rmk0})(iii) and (\ref{Ausobs}) that %$\pd( \up{\varphi^n}R)<\infty$, and hence
$R$ is regular; cf. Corollary \ref{Corintroo}. % and the paragraph following (\ref{corproof}).
%\end{enumerate}
\end{chunk}

One can find remarkable applications of (\ref{Ausobs}) in the literature. For example, Jorgensen \cite[2.1]{JorExt} proved that, if $R$ is a complete intersection ring and $M$ is an $R$-module such that $\Ext_{R}^2(M,M)=0$, then $\pd(M)\leq 1$; (\ref{Ausobs}) plays an important role in Jorgensen's proof. On the other hand, Dao exploited (\ref{Ausobs}) and obtained new results on the non-commutative crepant resolutions; see \cite{Da4} for details.

We give two applications of Auslander's rigidity result recorded in
(\ref{Ausobs}). The first one, (\ref{Gor}), is an immediate
observation, albeit it will be quite useful later; see the proof of
Corollary \ref{corgen}. Our second application is given in Corollary
\ref{corpr}: it yields a characterization of Cohen-Macaulay rings in
terms of Tor-rigidity. We proceed by recalling a remarkable result
of Foxby:

\begin{chunk} (Foxby \cite[3.1.25]{BH}) \label{Foxby} $R$ is Gorenstein if and only if there exists a nonzero finitely generated $R$-module $M$ such that $\pd(M)<\infty$ and $\id(M)<\infty$.
\end{chunk}

\begin{chunk} \label{Gor}
Let $M$ be a nonzero Tor-rigid $R$-module. If $\id(M)<\infty$, then it follows from (\ref{Ausobs}) that $\pd(M)<\infty$ and hence, by (\ref{Foxby}), $R$ is Gorenstein.
\end{chunk}

Tor-rigidity hypothesis in (\ref{Gor}) cannot be replaced with test property: a local ring admitting a nonzero test module of finite injective dimension is not necessarily Gorenstein, however such a ring $R$ is $G$-regular \cite{Takasha}, i.e., $\G-dim(M)=\pd(M)$ for all $R$-modules $M$; see (\ref{Test}), Example \ref{eg33} and Corollary \ref{Gregular}.

The \emph{grade} of a pair of nonzero modules $(M,N)$, denoted by $\gr(M,N)$, is defined as $\inf\{i\in \NN \cup \{0\}: \Ext^i_R(M,N)\neq0\}$. Setting $\gr(M)=\gr(M,R)$, we see that $\gr(M)<\infty$.

\begin{prop}\label{pr}  \label{corpr} Let $R$ be a local ring and let $M$ and $N$ be nonzero $R$-modules. Assume $N$ is Tor-rigid. Set $\gr(M)=n$ and $\gr(M,N)=s$. Then $s \leq n$, and $\Ext^i_R(M,N)\neq 0$ for all $i=s, \ldots, n$.
\end{prop}

\begin{proof} We have, by definition, that $\Ext^n_R(M,R)\neq 0$. If $\Ext^n_R(M,N)=0$, then it follows from (\ref{Ausobs}) that
$\Ext^n_R(M,R)=0$, which is a contradiction. Therefore $\Ext^n_R(M,N)\neq 0$, and hence $s\leq n$. Now suppose $\Ext^i_R(M,N)=0$ for some $s<i<n$. Set
%\begin{equation}\notag{}
$r=\min\{j\in \ZZ: \Ext^j_R(M,N)=0 \text{ with } s<j<n\}$.
%\end{equation}
We know, since $r<n$, that $\Ext^r_R(M,R)=0$. Therefore $\mathcal{T}_{r}M$ is stably isomorphic to $\Omega\mathcal{T}_{r+1}M$. Moreover it follows from (\ref{Ausobs}) that $\Tor_i^R(\mathcal{T}_{r+1}M,N)=0$ for all $i\geq 1$. This implies that
$\Tor_i^R(\mathcal{T}_rM,N)=0$ for all $i\geq 1$. Now we use
(\ref{AuBrsequence})(i) and deduce that $\Ext^{r-1}_R(M,N)=0$. This contradicts the choice of $r$, and finishes the proof.
\end{proof}

An immediate consequence of Proposition \ref{pr} is a characterization of local rings:

\begin{cor} \label{rigidcons} Let $R$ be a local ring and let $N$ be a nonzero Tor-rigid $R$-module. 
\begin{enumerate}[\rm(i)]
\item $\depth(N)\leq \dim(M)+\gr(M)$ for all nonzero $R$-modules $M$.
\item $\depth(N) \leq \depth(R)$.
\item If $N$ is maximal Cohen-Macaulay, then $R$ is Cohen-Macaulay.
\end{enumerate}
\end{cor}

\begin{proof}
It follows from Proposition \ref{pr} that $\gr(M,N)\leq\gr(M)$. On the other hand $\depth(N)-\dim(M)\leq \gr(M,N)$; see \cite[2.1(a)]{GradeG}. This implies $\depth(N)\leq \dim(M)+\gr(M)$, and justifies (i). In particular, if $M=k$, we deduce that $\depth(N) \leq \dim(k)+\gr(k)=\depth(R)$ and hence (ii) follows. Note that (iii) is an immediate consequence of (ii). 
\end{proof}

Alexandre Tchernev \cite{LLA} informed us that the conclusion of Corollary (\ref{rigidcons})(iii) can be obtained without using Proposition \ref{pr}. Next is Tchernev's argument; it is included here with his permission.

\begin{thm} Let $R$ be a $d$-dimensional local ring. Assume there exists a maximal Cohen-Macaulay $R$-module $M$ that is Tor-rigid. Then $R$ is Cohen-Macaulay.
\end{thm}

\begin{proof} (A. Tchernev) Pick a system of parameters $X=(x_1,\ldots,x_d)$ of $R$ that is an
$M$-regular sequence. Let $F=( F_i, f_i )$ be a free resolution of $M$ and let $K(X, R)$ be the Koszul complex of $R$ with respect to $X$.
Consider the first quadrant double complex $F_p\otimes K_q(X,R)$. The first filtration has $E^2$ term equal to the iterated homology 
modules  $E_{p,q}^2 = \HH_p(F, \HH_q(X))=  \Tor^R_p(  M,  \HH_q(X)  )$. On the other hand, second filtration has $E^2$ term equal to the
iterated homology modules $\HH_p(X, \HH_q(F))=\HH_p(X,M)$. Thus there is a first quadrant spectral sequence:
$$\Tor^R_p(M,\HH_q(X))\Longrightarrow \HH_{p+q}(X,M).$$
As $x_1,\cdots,x_d$ is a regular sequence on $M$, it follows that $\HH_i(X,M)=0$ for all $i\geq 1$. Furthermore $\Tor^R_1(M,\HH_0(X))=\Tor^R_1(M,R/X)=0$.
Since $M$ is Tor-rigid, we have that $\Tor^R_i(M,R/X)=0$ for all $i\geq 1$. Note  $d^2_{p,q}:E^2_{p,q}\longrightarrow E^2_{p-2,q+1}$ is of bidegree $(-2,1)$. Thus
$E^{\infty}_{0,1}=E^2_{0,1}=M\otimes_R \HH_1(X)$, which is a subquotient of $\HH_1(X,M)=0$.  It follows that $\HH_1(X)=0$ and so $E^2_{p,q}=0$ for all $p$ and $q\leq 1$, except $E^2_{0,0}$. 
Thus $E^{\infty}_{0,2}=E^2_{0,2}= M\otimes_R \HH_2(X)$, which is a subquotient of $\HH_2(X,M)=0$. Hence $\HH_2(X)=0$.
Proceeding in this way we see that $\HH_i(X)=0$ for all $i\geq 1$. Consequently $x_1,\cdots,x_d$ is an $R$-regular sequence, and $R$ is Cohen-Macaulay.
\end{proof}

It is an open question whether or not every complete Noetherian local ring has a maximal Cohen-Macaulay module: this is known as the small Cohen-Macaulay conjecture. However, in low dimensions, we can use Corollary \ref{rigidcons} and observe:

\begin{chunk} \label{MCMexamples} Assume $R$ is a local ring that is not Cohen-Macaulay.
\begin{enumerate}[\rm(i)]
\item If $R$ is one-dimensional and $\fp$ is a minimal prime ideal of $R$, then $R/\fp$ is a maximal Cohen-Macaulay $R$-module that is not Tor-rigid. For example if we put $R=k[\![x,y]\!]/(x^2, xy)$, then $R/(x)$ is not a Tor-rigid $R$-module. In fact  $\Tor^R_1(R/(x),R/(y))=0\neq \Tor^R_2(R/(x),R/(y))$; see  \cite[Question 3]{Li}.
\item If $R$ is a two-dimensional complete domain, then the integral closure $\overline{R}$ of $R$ in its field of fractions is a (finitely generated) maximal Cohen-Macaulay $R$-module that is not Tor-rigid. For example, if $R=k[\![x^4,x^3y,xy^3,y^4]\!]$, then $\overline{R}=R[\![x^2y^2]\!]$ is not a Tor-rigid $R$-module.
\end{enumerate}
\end{chunk}

One can also find examples of three dimensional non Cohen-Macaulay local rings that admit maximal Cohen-Macaulay modules; see, for example, Hochster \cite[5.4, 5.6 and 5.9]{Hochster1972}. These modules are not Tor-rigid, for example, by (\ref{rigidcons}).

Recall that an $R$-module module $C$ is called \emph{semidualizing} if the natural map $R \to \Hom(C,C)$ is bijective and $\Ext^{i}_{R}(C,C)=0$ for all $i\geq 1$. If $C$ is a semidualizing module such that $\id(C)<\infty$, then $R$ is Cohen-Macaulay and $C$ is dualizing.

\begin{chunk} \label{Cpd} Let $C$ be a semidualizing module over $R$. Then the
\emph{$C$-projective dimension} $\cpd(M)$ of a nonzero $R$-module $M$ is defined as the infimum of the integers $n$ such that there exists an exact sequence
$$0\to C^{b_n} \to C^{b_{n-1}} \to \cdots \to C^{b_1} \to C^{b_0} \to M \to 0$$
where each $b_i$ is a positive integer. It follows $\cpd(M)=\pd(\Hom_{R}(C,M))$, and the \emph{$C$-injective dimension} $\cid(M)$ of $M$ is defined similarly: $\cid(M)=\id(C\otimes_{R}M)$; see \cite[1.6, 2.8, 2.9 and 2.11]{TW}. Notice, if $\pd(C)<\infty$, then $C\cong R$ and hence $\cpd(N)=\pd(N)$ and $\cid(N)=\id(N)$.
\end{chunk}

\begin{chunk} (Takahashi and White \cite[2.9]{TW}) \label{T-W} Let $M$ be a nonzero $R$-module and let $C$ be a semidualizing $R$-module. If $\cpd(M)<\infty$ (respectively, $\cid(M)<\infty$), then $\Ext^i_R(C,M)=0$ for all $i\geq 1$ (respectively, $\Tor_i^R(C,M)=0$ for all $i\geq 1$). In particular, if $M$ is a nonzero test module and $\cid(M)<\infty$ for some semidualizing $R$-module $C$, then $C \cong R$ and hence $\id(M)<\infty$; see (\ref{Torrigid}) and (\ref{Cpd}).
\end{chunk}

\begin{prop} \label{rmk2} \label{corgen} Let $R$ be a local ring and let $C$ be a semidualizing $R$-module. Assume $M$ is a nonzero test module over $R$. Assume further that $\cid(M)<\infty$. If $X$ is an $R$-module such that $\Ext^{i}_{R}(X,R)=0$ for all $i\gg 0$, then $\pd(X)<\infty$.
\end{prop}

\begin{proof} Assume $X$ is an $R$-module with $\Ext^{i}_{R}(X,R)=0$ for all $i\gg 0$. Note, by (\ref{T-W}), we have that $\id(M)<\infty$. This yields $\Rhom(\Rhom(X,R),M) \simeq X\tensor_{R} M$; see \cite[A.4.24]{Larsbook}. Therefore $\Tor_{i}^{R}(M,X)=0$ for all $i\gg 0$ so that $\pd(X)<\infty$. %This, in particular, implies that $R$ is $G$-regular.
\end{proof}

It was proved in \cite[3.7]{CDtest} that, if the dualizing module of a Cohen-Macaulay local ring $R$ is a test module, then $R$ is $G$-regular, i.e.,
$\pd(M)=\G-dim(M)$ for all $R$-modules $M$ \cite{Takasha}. A straightforward application of Proposition \ref{rmk2}  extends this:

\begin{cor} \label{Gregular} A local ring admitting a nonzero test module of finite injective dimension  is $G$-regular.
\end{cor}

Before we proceed to prove our main result, we give an overview of what has been established so far in terms of the injective dimension of test and rigid modules:

\begin{chunk} Let $R$ be a local ring and let $N$ be a nonzero $R$-module such that $\id(N)<\infty$.
\begin{enumerate}[\rm(i)]
\item If $N$ is Tor-rigid, then $R$ is Gorenstein; see (\ref{Gor}).
\item If $N$ is a test module over $R$, then $R$ is G-regular; see Corollary \ref{Gregular}.
\item If $N$ is a rigid-test module over $R$, then it follows from (i) and (ii) that $R$ is regular; see also Question \ref{qintro} and Corollary \ref{cpcidcor}.
\end{enumerate}
\end{chunk}

\section{Main theorem} This section is dedicated to a proof of our main result, Theorem \ref{lemImp}. In the following $\H$ denotes a homological dimension of finitely generated modules; see (\ref{hdim}) and, for example, Avramov's expository article \cite[8.6 - 8.8]{Luchosurvey} for details. The special case -- where $\HH$ is the projective dimension $\pd$ -- is what we really need for the proof of Theorem \ref{thmint}, stated in the introduction. However one can follow our argument word for word by replacing $\H$ with projective dimension $\pd$ so there is no extra penalty for this generality. Furthermore such a generality is useful to examine the \emph{Gorenstein dimension} $\G-dim$ of Tor-rigid modules; see Corollary \ref{corofthm0}.

\begin{chunk} (\cite{AuBr}) A finitely generated module $M$ over a commutative Noetherian ring $R$ is said to be \emph{totally reflexive} if the canonical map $M \to \Hom(\Hom(M,R),R)$ is bijective, and $\Ext^i_{R}(M,R) = 0 = \Ext^i_{R}(M,\Hom(M,R))=0$ for all $i\geq 1$.

The infimum of nonnegative integers $n$ for which there exists an exact sequence
$0 \to X_{n} \to \dots  \to X_{0} \to M\to 0,$
such that each $X_{i}$ is totally refleive, is called the Gorenstein dimension of $M$. If $M$ has Gorenstein dimension $n$, we write $\G-dim(M)=n$. Therefore $M$ is totally reflexive if and only if $\G-dim(M)\leq 0$, where it follows by convention that $\G-dim(0)=-\infty$.
\end{chunk}

\begin{chunk}  \label{hdim} Throughout we assume $\H$ satisfies the following conditions: %$\phantom{}$
\begin{enumerate}[\rm(i)]
\item $\G-dim(M)\leq \H(M)\leq \pd(M)$ for all $R$-modules $M$.
\item If $\H(M)=0$, then $\Tr M = 0$ or $\H(\Tr M)=0$ for all $R$-modules $M$.
\end{enumerate}
\end{chunk}

Although we will not use it, we note that the \emph{complete intersection dimension}  \cite{AGP} is an example of a homological dimension -- in general distinct than the Gorenstein and projective dimension -- that satisfies the conditions in (\ref{hdim}).

\begin{chunk} \label{purp} For our purpose we recall a few properties of $\H$; see \cite[3.1.2, 8.7 and 8.8]{Luchosurvey}.
\begin{enumerate}[\rm(i)]
\item If one of the dimensions in (\ref{hdim})(i) is finite, then it equals the one on its left.
\item If $\H(M)<\infty$, then $\H(M)=\sup \{i \in \ZZ: \Ext^i_R(M,R) \neq 0\}$.
\item If $\H(M)<\infty$, then $\H(M)\leq \depth(R)$.
\end{enumerate}
\end{chunk}

If $X$ and $Y$ are nonzero $R$-modules and $\H$ is a homological dimension of modules, we consider the following condition for $(X, Y, \HH)$:

\begin{chunk} \label{dfcond}
If $\Tor_{1}^{R}(X,Y)=0$, then $\H(X)=\depth(Y)-\depth(X\otimes_{R}Y)$.
\end{chunk}

The condition in (\ref{dfcond}) is not restrictive for rigid modules. For example, Auslander  \cite[1.2]{Au} proved that, if $R$ is a local ring, and $X$ and $Y$ are nonzero $R$-modules where $\pd(X)<\infty$ and $Y$ is Tor-rigid, then (\ref{dfcond}) holds for $(X, Y, \pd)$; see (\ref{Torrigid}). 

Recently Christensen and Jorgensen \cite[5.3]{CJ} established a similar result over AB rings: if $R$ is AB, and $X$ and $Y$ are nonzero $R$-modules either of which is Tor-rigid, then $(X, Y, \G-dim)$ satisfies the condition in (\ref{dfcond}). Recall that a Gorenstein local ring $R$ is said to be AB \cite{HJ} if, for all $R$-modules $M$ and $N$, $\Ext^i_{R}(M,N)=0$ for all $i\gg 0$ implies that $\Ext^i_{R}(M,N)=0$ for all $i> \dim(R)$. The class of AB rings strictly contain that of complete intersections; see \cite[3.6]{HJ} and \cite[4.5]{JorSega} for details.

Next we summarize the aferomentioned two results:

\begin{chunk} \label{conds} Let $R$ be a local ring and let $X$ and $Y$ be nonzero $R$-modules. Then $(X, Y, \HH)$ satifies the condition in (\ref{dfcond}) if at least one of the following conditions holds:
\begin{enumerate}[\rm(i)]
\item $Y$ is a rigid-test module, and $\H=\pd$; see (\ref{testrigid}) and \cite[1.2]{Au}.
\item $R$ is AB, $X$ or $Y$ is Tor-rigid, and $\H=\G-dim$; see (\ref{Torrigid}) and \cite[5.3]{CJ}.
\end{enumerate}
\end{chunk}

Next is the key result we use for our proof of Theorem \ref{lemImp}; see also (\ref{AuBrsequence}).

\begin{chunk} \label{AUBR} (Auslander and Bridger; see \cite[2.12, 2.15 and 2.17]{AuBr}) Let $R$ be a local ring, $M$ a nonzero $R$-module and let $n$ be a positive integer. If $M$ is $n$-torsion-free, then it is $n$-reflexive, i.e., if $\Ext^i_R(\Tr M,  R)=0$ for all $i=1, \ldots, n$, then $M \approx \Omega^{n}_{R}(\mathcal{T}_{n+1}(\Tr_{R} M))$, i.e., $M$ is isomorphic to $\Omega^{n}_{R}(\mathcal{T}_{n+1}(\Tr_{R} M))$ up to a projective summand.
\end{chunk}

%The following condition will be used for Theorem \ref{lemImp}. 
Recall that $\depth(0)=\infty$.

\begin{chunk}  (\cite{AuBr}) \label{Sn} Let $R$ be a local ring and let $M$ be an $R$-module. Then  $M$ is said to satisfy $(\widetilde{\S}_n)$ if $\depth_{R_{\frak{p}}}(M_{\frak{p}})\geq \min\{n, \depth(R_{\frak{p}})\}$ for all $\frak{p}\in \Supp(M)$.

If the ring is Cohen-Macaulay, $(\widetilde{\S}_n)$ coincides with Serre's condition $(\S_n)$ \cite{EG}, but in general $(\widetilde{\S}_n)$ is a weaker condition. For example, if $R=k[\![x,y]\!]/(x^2, xy)$, then, by definition, $R$ satisfies $(\widetilde{\S}_n)$ for all nonnegative integers $n$, but fails to satisfy $(S_{1})$ since $\depth(R)=0$; see also the discussion following \cite[Definition 10]{Mas}.
\end{chunk}

Theorem \ref{lemImp} is a generalization of a result of Jothilingham
\cite[Corollary 1]{Jot2}. We give the argument in two steps as the
conclusion of part (1) may be of independent interest. It is already known that the vanishing of
$\Ext^i_R(\Tr M,R)$ for all $i=1, \dots, r$ forces $M$ to be an
$r$th syzgy module, and that forces $M$ to satisfy
$(\widetilde{\S}_r)$; see (\ref{Sn}) and \cite[Propositions 11 and
40]{Mas}.

\begin{thm} \label{lemImp} Let $R$ be a local ring and let $M$ and $N$ be a nonzero $R$-modules. Assume $M$ satisfies $(\widetilde{\S}_r)$ for some nonnegative integer $r$.
\begin{enumerate}[\rm(1)]
\item If $\G-dim(\Tr M)<\infty$, then $\Ext^i_R(\Tr M,R)=0$ for all $i=1, \dots, r$.
\item Assume $\Ext^n_R(M,N)=0$ for some positive integer $n$.
\begin{enumerate}[\rm(i)]
\item If $N$ is strongly-rigid and $\depth(R)\leq n+r$, then $\pd(M)\leq n-1$.
\item If $N$ is Tor-rigid, $\depth(N)\leq n+r$ and $(\mathcal{T}_{n+1}M, N, \HH)$ satisfies the condition in (\ref{dfcond}), then $\H(M)\leq n-1$.
\end{enumerate}
\end{enumerate}
\end{thm}

\begin{proof} We proceed to prove (1). Assume $\G-dim(\Tr M)<\infty$ and $r\geq 1$.
If $\G-dim(\Tr M)=0$, then $\Ext^i_R(\Tr M,R)=0$ for all $i\geq 1$ so that there is nothing to prove. So we assume $\G-dim(\Tr M)\geq 1$ and set $s=\inf\{i\geq 1: \Ext^i_R(\Tr M,R)\neq 0\}$.

Let $\frak{p}\in\Ass_R(\Ext^s_R(\Tr M,  R))$. Then we have $\frak{p}R_{\frak{p}} \in \Ass_{R_\frak{p}}(\Ext^s_{R_\frak{p}}(\Tr_{R_\frak{p}} M_\frak{p}, R_\frak{p}))$ and that $s=\inf\{i\geq 1: \Ext^i_{R_\frak{p}}(\Tr_{R_\frak{p}} M_\frak{p},R_\frak{p})\neq 0\}$. It follows from the definition of the transpose, see (\ref{AuBrsequence}), that there is an injection
\begin{equation} \notag{}
0\rightarrow\Ext^s_{R_\frak{p}}(\Tr_{R_\frak{p}} M_\frak{p}, R_\frak{p})\hookrightarrow \mathcal{T}_s(\Tr_{R_\frak{p}}M_\frak{p}),
\end{equation}
which shows that $\frak{p}R_{\frak{p}} \in \Ass_{R_\frak{p}}(\mathcal{T}_s(\Tr_{R_\frak{p}} M_\frak{p}))$. Hence $\depth_{R_\frak{p}}(\mathcal{T}_s(\Tr_{R_\frak{p}}M_\frak{p}))=0$. Since $\Ext^i_{R_\frak{p}}(\Tr_{R_\frak{p}}M_\frak{p}, R_\frak{p})=0$ for all $i=1, \ldots, s-1$, we conclude from (\ref{AUBR}) that
\begin{equation} \notag{}
M_\frak{p} \approx \Omega^{s-1}_{R_\frak{p}}(\mathcal{T}_{s}(\Tr_{R_\frak{p}} M_\frak{p})).
\end{equation}
Note that $s \leq \G-dim_{R_\frak{p}}(\Tr_{R_\frak{p}}M_\frak{p}) \leq \depth R_\frak{p}$.
Therefore $\depth_{R_\frak{p}}(M_\frak{p})=s-1$. Furthermore, since $M$ satisfies
 $(\widetilde{S}_{r})$, it follows that
$\depth_{R_\frak{p}}(M_\frak{p})\geq\min\{r,\depth R_\frak{p}\}$. Hence $\depth R_\frak{p}\geq r+1$ and $r\leq s-1$. Consequently $\Ext^i_R(\Tr M, R)=0$ for all
$i=1, \ldots, r$.

We now proceed to prove (2). If $\mathcal{T}_{n+1}M=0$, then $\Omega^nM$ is
free and hence $\pd(M)\leq n-1$; see (\ref{AuBrsequence}). In particular, this implies that $\H(M)\leq n-1$; see (\ref{purp})(i). Therefore we may assume $\mathcal{T}_{n+1}M\neq 0$.

Since $M$ satisfies $(\widetilde{S}_{r})$, it follows that $\Omega^{n}M$ satisfies $(\widetilde{S}_{n+r})$; see (\ref{Sn}). Therefore, by the first part of the theorem, we conclude that:
\begin{equation} \tag{\ref{lemImp}.1}
\Ext^i_R(\mathcal{T}_{n+1}M,R)=0 \text{ for all } i=1, \dots, n+r.
\end{equation}

Notice, since $\Ext^n_R(M,N)=0$, we have by (\ref{AuBrsequence})(i) that:
\begin{equation} \tag{\ref{lemImp}.2}
\Tor_1^R(\mathcal{T}_{n+1}M,N)=0.
\end{equation}

If (i) holds, then it follows from (\ref{stronglyrigid}) and
(\ref{lemImp}.2) that $\pd(\mathcal{T}_{n+1}M)<\infty$. Therefore,
since $\depth(R)\leq n+r$, we use (\ref{purp})(i) and deduce:
\begin{equation} \tag{\ref{lemImp}.3}
\H(\mathcal{T}_{n+1}M)=\pd(\mathcal{T}_{n+1}M)\leq \depth(R) \leq n+r.
\end{equation}

On the other hand, if (ii) holds, then it follows from our assumption that $\H(\mathcal{T}_{n+1}M)=\depth(N)-\depth(\mathcal{T}_{n+1}M\otimes_RN)$. Since $\depth(N)\leq n+r$, we obtain:
\begin{equation} \tag{\ref{lemImp}.4}
\H(\mathcal{T}_{n+1}M)\leq n+r.
\end{equation}
Consequently, if either (i) or (ii) holds, then $\H(\mathcal{T}_{n+1}M)\leq n+r$, where for part (i), $\H(\mathcal{T}_{n+1}M)=\pd(\mathcal{T}_{n+1}M)$; see (\ref{lemImp}.3) and (\ref{lemImp}.4).

Recall that $\H(\mathcal{T}_{n+1}M)=\sup\{i: \Ext^i_R(\mathcal{T}_{n+1}M,R)\neq0\}<\infty$; see (\ref{purp})(ii). Since $\H(\mathcal{T}_{n+1}M)\leq n+r$, it follows from (\ref{lemImp}.1) that $\H(\mathcal{T}_{n+1}M)=0$, i.e., $\H(\Tr \Omega^{n}M)=0$; see (\ref{AuBrsequence}). Thus, by (\ref{hdim})(ii), we have that $\H(\Omega^n M)=0$. This yields that $\H(M)\leq n$.

If (i) holds, then, since $\Ext^{n}(M,N)=0$, we conclude that $\H(M)\leq n-1$; see, for example, \cite[Chapter 19, Lemma 1(iii)]{Mat}. On the other hand, if (ii) holds, then, since $N$ is Tor-rigid, it follows from (\ref{lemImp}.2) and (\ref{AuBrsequence})(i) that $\Ext^n(M,R)=0$. Therefore we see that $\H(M)\leq n-1$.
\end{proof}

\section{Corollaries of the main theorem}

In this section we give various applications of Theorem \ref{lemImp} and examine homological dimensions of test and rigid-test modules. Corollary \ref{corofthm-1}, a reformulation of Theorem \ref{lemImp}, is fundamental to our work: it shows that one can use an arbitrary nonzero strongly-rigid, or a rigid-test module $N$, just like the residue field $k$, to determine the exact value of the projective dimension of $M$ via the vanishing of $\Ext_{R}^i(M,N)$. Besides this, Corollary \ref{corofthm-1} yields a series of related results. Among those is Corollary \ref{corsonuc1} which proves, in particular, that $R$ is Gorenstein if the Gorenstein injective dimension \cite{EJ} of the maximal ideal $\fm$ is finite.

Corollary \ref{corofthm-1} is well-known for the special case where $N=k$. Recall that a rigid-test module is, by definition, strongly-rigid, but we do not know whether or not all strongly-rigid modules are rigid-test; see Question \ref{qintro}.

\begin{cor} \label{corofthm-1} Let $R$ be a local ring, and let $M$ and $N$ be nonzero $R$-modules.
\begin{enumerate}[\rm(i)]
\item Assume $\Ext^n_{R}(M,N)=0$ for some $n\geq\depth(R)$. Assume further $N$ is strongly-rigid.
 Then $\pd(M)=\sup\{i\in \ZZ:  \Ext^i_{R}(M,N)\neq 0\}\leq n-1$.
\item Assume $\Ext^n_{R}(M,N)=0$ for some integer $n\geq \depth(N)$. Assume further $N$ is a rigid-test module. Then $\pd(M)=\sup\{i\in \ZZ:  \Ext^i_{R}(M,N)\neq 0\}\leq n-1$.
\end{enumerate}
\end{cor}

\begin{proof} 

Note that, for part (i) and part (ii), it suffices to prove that $\pd(M)$ cannot exceed $n-1$; see, for example, \cite[Chapter 19, Lemma 1(iii)]{Mat}.

Assume (i). If $n=0$ then $\depth(R)=0$. Since $\Hom_R(M,N)=0$ and $N$ is strongly-rigid, it follows from (\ref{rigidcons}) that $\depth(R)\geq\depth(N)\geq 1$, which is a contradiction.
Hence $n\geq 1$. Setting $r=0$ in Theorem \ref{lemImp}(2)(i), we conclude that $\pd(M)\leq n-1$. Next assume (ii). Notice, by (\ref{sifir}), $n$ is a positive integer. Moreover, since $N$ is a rigid-test module, (\ref{dfcond}) holds for $(\mathcal{T}_{n+1}M, N, \pd)$; see (\ref{conds})(i). Therefore we obtain the required conclusion by setting $r=0$ in Theorem \ref{lemImp}(2)(ii).
\end{proof}

The conclusion of Corollary \ref{corofthm-1} is sharp: Example \ref{eg2} shows that the condition on $n$ cannot be removed. Examples \ref{eg22} and \ref{eg33}, respectively, highlight the fact that the assumption ``$N$ is Tor-rigid" or ``$N$ is a test module" is not merely enough to deduce that $M$ has finite projective dimension; see (\ref{Torrigid}), (\ref{Test}) and (\ref{testrigid}).

\begin{eg} \label{eg2} Let $k$ be a field, $R=k[\![x,y,z]\!]/(xy-z^2)$, $M=k$ and $N=\Omega^2 k$. Then $N$ is a rigid-test module, $\Ext^1_{R}(M,N)=0$, $\depth(N)=2$ and $\pd(M)=\infty$
\end{eg}

\begin{eg} \label{eg22} Let $k$ be a field, $R=k[\![x,y]\!]/(xy)$, $M=k$ and $N=R/(x+y)$. Then, since $\pd(N)=1$, $N$ is Tor-rigid. Furthermore, since $R$ is not regular, $N$ is not a test module. Consequently $N$ is not a rigid-test, or a strongly-rigid module. Note that $\pd(M)=\infty$ and $\Ext^{i}_{R}(M,N)=0$ for all $i\geq 2$.
\end{eg}

\begin{eg} \label{eg33} Let $k$ be a field and put $R=k[\![x,y,z]\!]/(y^2-xz, x^2y-z^2, x^3-yz)$. Let $M=\fm$ and $N=\omega$, the canonical module of $R$. As $R$ is a one-dimensional domain with minimal multiplicity, it is Golod \cite[5.2.8]{Av2}. Hence, since $\pd(N)=\infty$, it follows from (\ref{onGolod}) that $N$ is a test-module. Huneke and Wiegand \cite[4.8]{HW1} proved that there exists an $R$-module $M$ such that $M\otimes_{R}N$ is torsion-free and $M$ has torsion. We now follow the proof of \cite[1.1]{HW1}:

Let $\overline{M}$ be the torsion-free part of $M$. Then, since $R$ is a domain, there is an exact sequence $0 \to \overline{M} \to F \to C \to 0$, where $F$ is a free $R$-module. Tensoring this short exact sequence with $N$, we obtain an injection $\Tor_1^R(C,N) \hookrightarrow \overline{M}\otimes_{R}N$. Since $\overline{M}\otimes_{R}N  \cong M\otimes_{R}N$ and $\Tor_1^R(C,N)$ is torsion, we see that $\Tor_1^R(C,N)=0$. Note that $\pd(C)=\infty$: otherwise $\overline{M}$ is free, and this would force $M$ to be free. Therefore $N$ is not a strongly-rigid module. For completeness, we also remark that $N$ is not Tor-rigid; see (\ref{Gor}). Notice $\Ext^{i}_{R}(M, N)=0$ for all $i\geq 1$ and $\pd(M)=\infty$.
\end{eg}

Corollary \ref{corofthm-1} can be useful to determine the depth of $\Hom(M,N)$:

\begin{cor} \label{propHom} Let $R$ be a Cohen-Macaulay local ring and let $M$ and $N$ be nonzero $R$-modules. Assume that the following conditions hold:
\begin{enumerate}[\rm(i)]
\item $R$ has an isolated singularity, i.e., $R_{\fp}$ is regular for all $\fp \in \Spec(R)-\{\fm\}$.
\item $M$ is nonfree and maximal Cohen-Macaulay.
\item $N$ is strongly-rigid with $\depth(N)\geq 2$.
\end{enumerate}
Then $\depth(\Hom_R(M,N))=2$.
\end{cor}

\begin{proof} Note that $\depth(\Hom_R(M,N))\geq \min\{2, \depth(N)\}=2$; see \cite[1.4.19]{BH}. Thus it suffices to prove $\depth(\Hom_R(M,N))\leq 2$. Assume not, i.e., assume $\depth(\Hom_R(M,N))\geq 3$. Then, by \cite[1.1]{Jinnah}, we see either $\Ext_R^1(M,N)=0$, or $1\leq \depth(\Ext_R^1(M,N))<\infty$. Since $\Ext_R^1(M,N)$ has finite length, it follows that $\Ext_R^1(M,N)=0$. Set $d=\dim(R)$. Then $M=\Omega^d(X)$ for some finitely generated $R$-module $X$; see \cite[A.15]{Book}. This yields $\Ext_R^{d+1}(X,N)=0$. Now Corollary \ref{corofthm-1} shows that $\pd(X)<\infty$, i.e., $M$ is free. So we conclude that $\depth(\Hom_R(M,N))=2$.
\end{proof}

\begin{cor} Let $A=Q/(f)$, where $Q=k[x_{1}, \dots, x_{2s+1}]$, with $s\geq 1$, is a polynomial ring over a perfect field $k$, and $f$ is a nonconstant polynomial in $Q$.
Set $R=A_{\fm}$, where $\mathfrak{m}=(x_{1}, \dots, x_{2s+1})A$ and assume $A_{p}$ is a regular for all $\fp \in \Spec(A)-\{\fm\}$. If $M$ and $N$ are nonfree maximal Cohen-Macaulay $R$-modules, then $\depth(\Hom_R(M,N))=2$.
\end{cor}

\begin{proof} In light of (\ref{onGolod}) and (\ref{Mark}), the required conclusion follows immediately from Corollary \ref{propHom}; cf. \cite[3.4]{Da3}.
\end{proof}

Corollary \ref{Corintroo}, stated in the introduction, is now a direct consequence of our argument and the following special case of \cite[1.1]{AHIY}; see also (\ref{rmk0})(i).

\begin{chunk} \label{AHiy} (Avramov, Hochster, Iyengar and Yao; see \cite[1.1]{AHIY}) Let $R$ be a local ring of prime characteristic $p$ and let $M$ be a nonzero finitely generated $R$-module. If $\up{\varphi^e}M$ has finite flat dimension for some positive integer $e$, then $R$ is regular.
\end{chunk}

\begin{cor} \label{CorintrooGen} Let $R$ be an F-finite local ring of prime characteristic $p$ and let $N$ be a nonzero rigid-test module over $R$. If $\Ext^j_{R}(\up{\varphi^n}M, N)=0$ for some nonzero $R$-module $M$, and for some integers $n\geq 1$ and $j\geq \depth(N)$, then $R$ is regular.
\end{cor}

\begin{proof} Notice, as $R$ is F-finite, $\up{\varphi^e}M$ is a finitely generated $R$-module. Thus it follows from Corollary \ref{corofthm-1}(ii) that $\pd(\up{\varphi^e}M) <\infty$. Now, by (\ref{AHiy}), $R$ is regular.
\end{proof}

\begin{chunk} \label{corproof} (Proof of Corollary \ref{Corintroo}) As $R$ is complete and $k$ is perfect, it follows that $R$ is F-finite; see, for example, \cite[page 398]{BH}. So the result follows from Corollary \ref{CorintrooGen}.
\end{chunk}

A special case of Corollary \ref{corofthm-1} and  \ref{CorintrooGen} has been established in \cite[Theorem B]{NMS}: if $R$ is a complete intersection ring of prime characteristic $p$, and $\Ext^n_{R}(M,\up{\varphi^i}R)=0$ for some $n\geq \depth(R)$, then $\pd(M)<\infty$; see (\ref{rmk0})(iii). We should note that \cite[Theorem B]{NMS} does not require an F-finite ring and relies on methods different from ours. As discussed in the introduction, our argument is not specific to rings of characteristic $p$, and gives useful information regarding the Frobenius endomorphism even if the ring considered is not a complete intersection. For example the next result, in view of (\ref{rmk0})(ii), is immediate from Corollary \ref{corofthm-1}; cf. \cite[Theorem A]{NMS}.

\begin{cor} \label{corofthm+1} Let $R$ be a one-dimensional F-finite Cohen-Macaulay local ring of prime characteristic $p$, and let $M$ be an $R$-module. Then $\Ext^n_{R}(M,\up{\varphi^i}R)=0$ for some $n\geq 1$ and some $i\gg 0$ if and only if $\pd(M)< \infty$.
\end{cor}

Corollary \ref{corofthm-1} yields a characterization of regularity in terms of $\cpd$ and $\cid$ dimensions of strongly-rigid modules; see also (\ref{stronglyrigid}) and (\ref{Cpd}).

\begin{cor} \label{cpcidcor} Let $R$ be a local ring, $C$ a semidualizing $R$-module and let $M$ be a nonzero strongly-rigid $R$-module. Assume either $\cpd(M)<\infty$ or $\cid(M)<\infty$. Then $R$ is regular.
\end{cor}

\begin{proof} We start by noting that $M$ is a test module; see  (\ref{Test}). Assume first $\cid(M)<\infty$. Then it follows from (\ref{T-W}) that $\id(M)<\infty$, i.e., $\Ext_R^i(k,M)=0$ for all $i\gg 0$. Now Corollary \ref{corofthm-1}(i) implies that $\pd(k)<\infty$ so that $R$ is regular.

Next assume $\cpd(M)<\infty$. Then it follows from (\ref{T-W}) that $\Ext^i_R(C,M)=0$ for all $i\geq 1$. Hence we can use Corollary \ref{corofthm-1}(i) once more and deduce that $\pd(M)<\infty$. This implies that $R$ is regular.
\end{proof}

A special case of Corollary \ref{cpcidcor} is a characterization of regularity in terms of integrally closed $\fm$-primary ideals; see (\ref{inttest}).

\begin{cor} Let $(R, \fm)$ be a local ring and let $I$ be an integrally closed $\fm$-primary ideal of $R$. Then $R$ is regular if and only if there exists a semidualizing $R$-module $C$ such that $\cid(\Omega^{n}I)<\infty$ or $\cpd(\Omega^{n}I)<\infty$ for some nonnegative integer $n$. In particular, $R$ is regular if and only if $\id(I)<\infty$.
\end{cor}

The conclusions of next corollaries, \ref{corofthm0} and \ref{corExt}, are known over complete intersection rings; see \cite[3.6]{Sa1}. Here we are able to show that these results carry over to AB rings. Recall that every complete intersection ring is AB, but not vice versa; see the paragraph preceding (\ref{conds}). Furthermore, in Corollary \ref{Extcor}, we obtain a nonvanishing result for Ext over hypersurfaces that are in the form of (\ref{Mark}).

\begin{cor} \label{corofthm0} Let $R$ be a local AB ring, and let $M$ and $N$ be nonzero $R$-modules. Assume $N$ is Tor-rigid and that $\Ext^n_{R}(M,N)=0$ for some $n\geq \depth(N)$. Then $\sup\{i\in \ZZ: \Ext^i_R(M,N)\neq 0\}=\G-dim(M)=\depth(R)-\depth(M)\leq n-1$.
\end{cor}

\begin{proof} The bound on $\G-dim(M)$ follows from Theorem \ref{lemImp}(ii): the conditions in (\ref{dfcond}) hold for $(\mathcal{T}_{n+1}M, N, \G-dim)$; see (\ref{conds})(ii).
Hence, since $R$ is an AB ring, it suffices to prove by \cite[3.2 and 6.1]{CJ} that $\Ext^i_R(M,N)=0$ for all $i\gg 0$.

As $\Ext^n_R(M,N)=0$ and $N$ is a Tor-rigid module, it follows from the exact
sequence (\ref{AuBrsequence})(i) that $\Tor_i^R(\mathcal{T}_{n+1}M,N)=0$ for all $i\geq 1$. This shows, by\cite[3.2]{CJ}, that Tate Tor groups $\widehat{\Tor}_i^R(\mathcal{T}_{n+1}M,N)$ vanish for all
$i\in\mathbb{Z}$. Note that $\G-dim(M)\leq n-1$, and $\G-dim(\Omega^nM)=0$; see \cite[3.13]{AuBr}. We now use
\cite[4.4.7]{AvBu} and conclude, for all $i\geq n+1$, that:
\[\begin{array}{rl}
\Ext^{i}_R(M,N) \cong\Ext^{i-n}_R(\Omega^{n}M,N) & \hspace{-0.1in} \cong \widehat{\Ext}^{i-n}_R(\Omega^nM,N))\\
& \hspace{-0.1in} \cong\widehat{\Tor}_{-i+n-1}^R(\Hom(\Omega^nM,R),N)\\
& \hspace{-0.1in} \cong\widehat{\Tor}_{-i+n+1}^R(\mathcal{T}_{n+1}M,N).
\end{array}\]
Therefore we have that $\Ext^i_R(M,N)=0$ for all $i\geq n$, and this proves our claim.
\end{proof}

\begin{cor} \label{corExt} Let $R$ be a local AB ring, and let $M$ and $N$ be nonzero $R$-modules such that $N$ is Tor-rigid (e.g., $N=k$).
Assume $\depth(N)\leq \G-dim(M)$. Then $\Ext^i_{R}(M,N)\neq 0$ for all $i$, where $\depth(N) \leq i \leq \G-dim(M)$. In particular, if $\depth(M)=0$, then $\Ext^i_{R}(M,N)\neq 0$ for all $i$, where $\depth(N)\leq i\leq \dim(R)$.
\end{cor}

\begin{proof} The first part is clear from Corollary \ref{corofthm0}. If $\depth(M)=0$, then $\G-dim(M)=\dim(R)$ so that the second part follows.
\end{proof}

If $R$ is a hypersurface (quotient of an equi-characteristic regular local ring) and $N$ is an $R$-module such that $\id(N)<\infty$, then $\pd(N)<\infty$ and hence it follows from a result of Lichtenabum \cite[Theorem 3]{Li} that $N$ is Tor-rigid. Consequently, when $R$ is such a hypersurface, and $M$ and $N$ are nonzero $R$-modules such that $\pd(M)<\infty$, $\depth(M)=0$ and $\id(N)<\infty$, Corollary \ref{corExt} implies that $\Ext^i_{R}(M,N)\neq 0$ for all $i$, where $\depth(N)\leq i\leq \dim(R)$. Over certain hypersurfaces, we know that all modules are Tor-rigid so that the nonvanishing of $\Ext^i_{R}(M,N)$ occurs without any restriction on $N$. For example, Corollary \ref{corExt} and (\ref{Mark}) yield:

\begin{cor} \label{Extcor} Let $A=Q/(f)$, where $Q=k[x_{1}, \dots, x_{2s+1}]$, with $s\geq 1$, is a polynomial ring over a perfect field $k$, and $f$ is a nonconstant polynomial in $Q$.
Set $R=A_{\fm}$, where $\mathfrak{m}=(x_{1}, \dots, x_{2s+1})A$ and assume $A_{p}$ is a regular for all $\fp \in \Spec(A)-\{\fm\}$. If $M$ and $N$ are nonzero $R$-modules such that $\depth(M)=0$, then $\Ext^i_{R}(M,N)\neq 0$ for all $i$, where $\depth(N)\leq i\leq \dim(R)$.
\end{cor}

\section{Gorenstein injective dimension of strongly-rigid modules} Let $(R, \fm,k)$ be a local ring. If $R$ is Gorenstein and the injective dimension $\id(\fm)$ of the maximal ideal $\fm$ is finite, so is the injective dimension of $k$, and hence the Auslander-Buchsbaum formula implies that $R$ is regular; see \cite[3.1.26]{BH}. Our results in this section originated in an attempt to answer the following question:

\begin{ques} \label{soru}
Let $(R, \fm)$ be a local ring. Assume $\id(\fm)<\infty$. Then must $R$ be Gorenstein, or equivalently, must $R$ be regular?
\end{ques}

Levin and Vasconcelos \cite[Theorem 1.1]{LV} proved that $R$ is regular if $\pd(\fm M)<\infty$ for an $R$-module $M$ with $\fm M\neq 0$. They also remarked that an argument analogous to that of \cite[Theorem 1.1]{LV} would work just as well for finite injective dimension.

Avramov \cite{LLA} pointed out that an affirmative answer to Question \ref{soru} came out in a discussion with himself and H.-B. Foxby in the summer of 1983. He also referred us to  Lescot's explicit computation of the Bass series of $\fm$ for an example of a published treatment of this fact; see \cite[1.8]{Lescot}. Avramov \cite[Theorem 4]{LAER} proved that any submodule $L$ of a finitely generated $R$-module $M$ satisfying $L \supseteq \mathfrak{m}M \supsetneq \mathfrak{m}L$ has the same injective complexity and curvature as the residue field $k$. It follows, for example, if $\fm^nM\neq 0$ and $\id(\fm^nM)<\infty$, then $R$ is regular; see also \cite[Corollary 5]{LAER} and the remark following it. A very special case of this result -- the case where $n=1$ and $M=R$, i.e., the case where $\id(\fm)<\infty$ -- also follows from (\ref{Gor}).

The Gorenstein injective dimension, introduced by Enochs and Jenda \cite{EJ}, is a refinement of the classical injective dimension. We use Corollary \ref{corofthm-1} and prove that $R$ is Gorenstein if the Gorenstein injective dimension of an integrally closed $\fm$-primary ideal of $R$ is finite; see Corollary \ref{corsonuc1}. This, in particular, refines Question \ref{soru} and establishes that $R$ is Gorenstein if and only if the Gorenstein injective dimension of the maximal ideal $\fm$ is finite. We proceed by recalling some definitions:

\begin{chunk} \label{tanim} (\cite{EJ}; see \cite[6.2.2]{Larsbook}) An $R$-module $M$ is said to be Gorenstein injective if there is an exact sequence
$I_{\bullet}=(\cdots \to I_{1} \stackrel{\rm \partial_{1}}{\longrightarrow} I_{0} \stackrel{\rm \partial_{0}}{\longrightarrow}  I_{-1} \to \cdots)$ of injective $R$-modules such that $M\cong \ker(\partial_{0})$ and $\Hom_{R}(E,I_{\bullet})$ is exact for any injective $R$-module $E$. The Gorenstein injective dimension of $M$, $\Gid(M)$, is defined as the infimum of $n$ for which there exists an exact sequence $0 \to M \to I_{0} \to \cdots \to I_{-n}\to0$, where  each $I_{i}$ is Gorenstein injective.

The Gorenstein injective dimension is a refinement of the classical injective dimension: $\Gid(M)\leq \id(M)$, with equality if $\id(M)<\infty$; see \cite[6.2.6]{Larsbook}. It follows that every module over a Gorenstein ring has finite Gorenstein injective dimension. Hence, if $(R, \fm)$ is Gorenstein, but not regular, then $\Gid(k)<\infty=\id(k)$.
\end{chunk}

Cohen-Macaulay local rings that admit a nonzero strongly-rigid module of finite Gorenstein injective dimension are Gorenstein.

\begin{prop} \label{sonuc} Let $R$ be a Cohen-Macaulay local ring. If $\Gid(M)<\infty$ for some nonzero strongly-rigid $R$-module $M$, then $R$ is Gorenstein.
\end{prop}

\begin{proof} Assume $\Gid(M)<\infty$ for some nonzero strongly-rigid $R$-module $M$. Then, since $R$ is Cohen-Macaulay, there exists a nonzero finitely generated $R$-module $N$ such that $\id(N)<\infty$ and $\Ext^{i}_{R}(N,M)=0$ for all $i\gg 0$; see \cite[2.22]{Henrik}. Therefore Corollary \ref{corofthm-1}(i) shows that $\pd(N)<\infty$, i.e., $R$ is Gorenstein; see (\ref{Foxby}). %This completes the proof; see (\ref{tanim}).
\end{proof}

In general it is not known whether or not a local ring admitting a nonzero module of finite Gorenstein injective dimension must be Cohen-Macaulay; see \cite{Larsbook}. Hence, in view of the foregoing result, it seems worth raising the following question:

\begin{ques} \label{CMsoru} Let $R$ be a local ring and let $M$ be a nonzero strongly-rigid module over $R$. If $\Gid(M)<\infty$, then must $R$ be Gorenstein?
\end{ques}

We are able to give an affirmative answer to Question \ref{CMsoru} when $M$ is an integrally closed $\fm$-primary ideal. For that we need the following result of Yassemi:

\begin{chunk} \label{yas0} (Yassemi \cite[1.3]{YasGdim}) Let $R$ be a local ring and let $M$ be a nonzero $R$-module. Assume $\Gid(M)<\infty$. Assume further that $\dim(M)=\dim(R)$. Then $R$ is Cohen-Macaulay.
\end{chunk}

An integrally closed $\fm$-primary ideal $I$ of $(R, \fm)$ is a strongly-rigid module with $\dim(I)=\dim(R)$; see (\ref{inttest}). So we deduce from Proposition \ref{sonuc} and  (\ref{yas0}) that:

\begin{cor} \label{corsonuc1} A local ring $(R, \fm)$ is Gorenstein  if and only if $\Gid(I)<\infty$ for some integrally closed $\fm$-primary ideal $I$ of $R$. In particular, $(R, \fm)$ is Gorenstein if and only if $\Gid(\fm)<\infty$.
\end{cor}

\begin{rmk} \label{Gidremark} It is known that if $\Gid(k)<\infty$ or $\Gid(R)<\infty$, then $R$ is Gorenstein; see \cite[6.2.7]{Larsbook} and \cite[2.1]{Henrik2}. However we do not know whether the finiteness of $\Gid(\fm)$ directly implies the finiteness of $\Gid(k)$ or $\Gid(R)$ via the short exact sequence $0\to \fm \to R \to k \to 0$. Thus, as far as we know, Corollary \ref{corsonuc1}, even for the special case where $I=\fm$, is new.
\end{rmk}

Our final aim is to show in Proposition \ref{sonuc2} that the finiteness of $\Gcid(\fm)$ for a semidualizing module $C$ detects the dualizing module, i.e., forces $C$ to be dualizing. We first record a few preliminary results.

\begin{chunk} \label{Cgid} Let $C$ be a semidualizing $R$-module; see (\ref{Cpd}). Then the \emph{$C$-Gorenstein injective dimension} $\Gcid(M)$ of a nonzero $R$-module $M$ can be defined as $\Gid_{R\ltimes C}(M)$, where $R\ltimes C$ is the trivial extension of $R$ by $C$; see Henrik and {J{\o}rgensen} \cite[2.16]{HHJP}. In particular, if $C=R$, then $\Gcid(M)=\Gid(M)$; see (\ref{tanim}).
\end{chunk}

Proposition \ref{corr0} is used for our proof of Proposition \ref{sonuc2} for the case where $R$ is Artinian; see also (\ref{stronglyrigid}).

\begin{prop} \label{corr0} Let $R$ be a Cohen-Macaulay local ring with a dualizing module and let $C$ be a  semidualizing $R$-module. Assume $M$ is a nonzero strongly-rigid module over $R$. Assume further $\Gcid(M)<\infty$. Then $C$ is dualizing.
\end{prop}

\begin{proof} Note that $\Ext^{i}_{R}(C,\omega)=0$ for all $i\geq 1$, where $\omega$ is the dualizing module. Hence $X=\Rhom_{R}(C,\omega)\simeq \Hom_{R}(C,\omega)$  is a maximal Cohen-Macaulay $R$-module. It follows from \cite[4.6]{HHJP} that $M$ is in the Bass class of $R$ with respect to $X$. In particular $\Ext^{i}_{R}(X, M)=0$ for all $i\gg 0$; see \cite[4.1]{HHJP}. Therefore Corollary \ref{corofthm-1}(i) implies that $C\cong \omega$.
\end{proof}

\begin{chunk} \label{Yass1} Let $M$ be a nonzero $R$-module and let $C$ be a  semidualizing $R$-module. Assume $\Gcid(M)<\infty$. If $\dim(M)=\dim(R)$, then $\dim_{R\ltimes C}(M)=\dim(R\ltimes C)$ so that, by (\ref{yas0}), $R$ is Cohen-Macaulay. Therefore, if $M$ is a strongly-rigid module, $\dim(M)=\dim(R)$ and $\depth(R)=0$, then $R$ is Artinian and it follows from Proposition \ref{corr0} that $C$ is dualizing.
\end{chunk}

For the rest of our arguments, $\overline{X}$ denotes $X/xX$, where $X$ is an $R$-module and $x$ is a non-zerodivisor on $R$.

\begin{chunk} \label{Serre} Let $(R, \fm)$ be a local ring and let $x\in\fm-\fm^2$ be a non-zerodivisor on $R$. Then the surjective $R$-linear map $f:\fm/x\fm \twoheadrightarrow \fm/xR$, given by  $f(y+x\fm)=y+xR$ for all $y\in \fm$, splits; see, for example, the proof of \cite[19.2]{Mat}. Therefore there exists an $\overline{R}$-module $N$ such that $\overline{\fm} \cong N\oplus \fm/xR$.
\end{chunk}

\begin{chunk} \label{modout} Let $R$ be a local ring and let $x\in R$. Assume $x$ is a non-zerodivisor on $R$. Then $x$ is also a non-zerodivisor on $R\ltimes C$, and hence the following holds:
\begin{equation} \tag{\ref{modout}.1}
\overline{R} \ltimes \overline{C} \cong \overline{R\ltimes C}
\end{equation}

Let $M$ be an $R$-module. Assume $x$ is also a non-zerodivisor on $M$. Assume further $\Gcid(M)<\infty$ for some semidualizing $R$-module $C$. Then, in view of (\ref{modout}.1), we deduce from \cite[Lemma 2]{SWY} that:
\begin{equation} \tag{\ref{modout}.2}
\Gid_{\overline{R\ltimes C}}(\overline{M})=\Gid_{\overline{R}\ltimes \overline{C}}(\overline{M})=\Gxcid_{\overline{R}}(\overline{M})<\infty.
\end{equation}
\end{chunk}

\begin{prop} \label{sonuc2} Let $R$ be a local ring. Assume at least one of the following conditions holds:
\begin{enumerate}[\rm(i)]
\item $\Gcid(\fm)<\infty$ for some semidualizing $R$-module $C$.
\item $\Gcid(M)<\infty$ for some maximal Cohen-Macaulay strongly-rigid module $M$.
\end{enumerate}
Then $C$ is dualizing.
\end{prop}

\begin{rmk} We already know from (\ref{Yass1}) that $R$ must be  Cohen-Macaulay in case (i) or (ii) holds in Proposition \ref{sonuc2}.
\end{rmk}

\begin{proof} [Proof of Proposition \ref{sonuc2}] We proceed by induction on $d=\depth R$. If $d=0$, then (\ref{Yass1}) gives the required conclusion for both case (i) and (ii). Hence we assume $d\geq 1$ and pick a non-zerodivisor $x$ on $R$ such that $x\in\fm-\fm^2$.

First suppose (i) holds, i.e., $\Gcid(\fm)<\infty$. It follows from (\ref{Serre}) that there exists an $\overline{R}$-module $N$ such that $\overline{\fm} \cong N \oplus\fm/xR$. Therefore we conclude from  (\ref{modout}.2) that $\Gid_{\overline{R\ltimes C}}(N \oplus\fm/xR)=\Gid_{\overline{R\ltimes C}}(\overline{\fm})<\infty$. Then \cite[2.6]{Henrik} implies that
$\Gid_{\overline{R\ltimes C}}(\fm/xR)<\infty$. Hence we obtain the following; see also (\ref{Yass1}) and (\ref{modout}.1).
\begin{equation} \notag{}
\Gid_{\overline{R\ltimes C}} (\fm/xR) = \Gid_{\overline{R} \ltimes \overline{C}}(\fm/xR)=\Gxcid_{\overline{R}}( \fm/xR)<\infty.
\end{equation}
Now the induction hypothesis forces $\overline{C}$ to be dualizing over $\overline{R}$, i.e., $\id_{\overline{R}}(\overline{C})<\infty$. Consequently $\id_{R}(C)<\infty$ and hence $C$ is dualizing over $R$.

Next assume (ii). Then it follows from (\ref{modout}.2) that $\Gxcid_{\overline{R}}(\overline{M})<\infty$. Moreover, since $M$ is maximal Cohen-Macaulay, $x$ is a  non-zerodivisor on $M$. We can easily observe, similar to \cite[2.2]{CDtest}, that $\overline{M}$ is strongly-rigid over $\overline{R}$: here we include an argument since \cite{CDtest} deals with only test modules; see also (\ref{Test}) and (\ref{stronglyrigid}).

Suppose $\Tor_{n}^{\overline{R}}(\overline{M},L)=0$ for some $\overline{R}$-module $L$, and for some positive integer $n$. Then, since $M$ is a strongly-rigid module over $R$, and  $\Tor_{i}^{\overline{R}}(\overline{M},L)\cong \Tor_{i}^{R}(M,L)$ for all $i\geq 0$, we see that $\pd_{R}(L)<\infty$. Now the fact $x\notin \mathfrak{m}^{2}$ implies that $\pd_{\overline{R}}(L)<\infty$; see \cite[3.3.5(1)]{Av2}. This proves that $\overline{M}$ is strongly-rigid over $\overline{R}$. Thus the induction hypothesis implies that $\overline{C}$ is dualizing over $\overline{R}$, and so $C$ is dualizing over $R$.
\end{proof}

Proposition \ref{sonuc2} naturally raises the following question; see (\ref{stronglyrigid}).

\begin{ques} Let $(R, \fm)$ be a local ring and let $C$ be a semidualizing $R$-module. If $\Gcid(M)<\infty$ for some nonzero strongly-rigid $R$-module $M$, then must $R$ be Cohen-Macaulay? What if $M$ is an integrally closed $\fm$-primary ideal?
\end{ques}

The \emph{Cohen-Macaulay injective dimension} $\CMid(M)$ \cite[2.3]{HJCMid} of an $R$-module $M$ is defined as $\inf\{\Gcid(M):\;$C$ \text{ is a semidualizing $R$-module} \}$; see also (\ref{Cgid}). %Holm and Foxby \cite[4.6]{} proved that, if $\CMid(M)<\infty$ for a cyclic module $M$ over a local ring $R$, then $R$ is Cohen-Macaulay with a dualizing module.
In view of this notation, Proposition \ref{sonuc2} characterizes Cohen-Macaulay rings in terms of the finiteness of the Cohen-Macaulay injective dimension of the maximal ideal $\fm$, i.e., if $\CMid(\fm)<\infty$, then $R$ is Cohen-Macaulay with a dualizing module $C$.

\appendix

\section{Some examples of test and rigid-test modules}
There are quite a few examples of test and rigid-test modules in the literature. In this section we catalogue a few of them; see also \cite[1.4]{CDtest} for a characterization of test modules over complete intersection rings. We start by pointing out that test and Tor-rigid modules are distinct in general; see (\ref{Torrigid}) and (\ref{Test}). 

\begin{chunk} \label{withLong} Let $S=k[\![x,y,z]\!]$ be the formal power series over a field $k$, and let $R$ to be the subring of $S$ generated by monomials of degree $2$, that is, the $2$nd \emph{Veronese} subring $S$. Then $S=R \oplus_{R}L$, where $L$ generates the class group of $R$; see, for example \cite[3.16]{CeD2}. Since $S$ is a finite extension of $R$, \cite[2.4]{CDtest} shows that $L$ is a test module over $R$. %Therefore $L$ is a test module over $R$.
On the other hand, %since $R$ is a three-dimensional Cohen-Macaulay ring with an isolated singularity, it satisfies %Serre's conditions $(R_2)$ and $(S_3)$. Hence it follows from
by Dao's remark \cite[2.6]{Da4}, $L$ is not Tor-rigid.
\end{chunk}

Recall that all rigid-test modules are strongly-rigid; see (\ref{testrigid}) and (\ref{stronglyrigid}).

\begin{chunk} \label{inttest} If $I$ is an integrally closed $\mathfrak{m}$-primary ideal of $R$ and $\Tor^{R}_{n}(R/I, N)=0$, then $\pd(N)\leq n-1$, i.e., $\Omega^{i}(R/I)$ is a rigid-test module for all $i\geq 1$; see \cite[3.3]{CHKV}.
\end{chunk}

\begin{chunk} \label{rmk0} $\phantom{}$ Let $R$ be an F-finite local ring of prime characteristic $p$, and let $\varphi^n: R \to R$ be the nth iterate of the Frobenius endomorhism defined by $r\mapsto r^{p^n}$ for $r\in R$. If $M$ is a finitely generated $R$-module, $\up{\varphi^n}M$ denotes the (finitely generated) $R$-module $M$ with the $R$-action given by $r\cdot m=\varphi^n(r)m$.
\begin{enumerate}[\rm(i)]
\item $\up{\varphi^n}R$ is a test module over $R$ for all $n\gg 0$; see (\ref{Test}) and \cite[2.2.8]{Msurvey}.
\item If $R$ is a one-dimensional Cohen-Macaulay local ring, then $\up{\varphi^n}R$ is a rigid-test module over $R$ for all $n\gg 0$; see (\ref{testrigid})  and \cite[2.1.3 and 2.2.12]{Msurvey}.
\item If $\fm^{[p]}=0$, or $R$ is a complete intersection ring, then $\up{\varphi^n}R$ is a rigid-test module over $R$ for all $n\geq 1$; see (\ref{testrigid}) and \cite[2.2.9, 2.2.10 and 5.1.1]{Msurvey}.
\end{enumerate}
\end{chunk}

\begin{chunk} \label{onGolod} Let $R$ be a Golod ring (e.g., $R$ is a hypersurface). If $M$ is an $R$-module with $\pd(M)=\infty$, then $M$ is a test-module over $R$; see \cite[section 5]{Av2} and \cite[3.1]{JAB}.
\end{chunk}

\begin{chunk} \label{Gsing} Let $(R, \fm,k)$ be a two-dimensional complete normal local domain with an algebraically closed residue field $k$. Assume $R$ has a rational singularity, i.e., there exists a resolution of singularities $X \to \Spec(R)$, a proper birational morphism where $X$ is a regular scheme, such that $\HH^1(X, O_X)=0$; see \cite[6.32]{Book} or \cite{Lipman}. It follows that $R$ is a Cohen-Macaulay ring with minimal multiplicty; see, for example, \cite[6.36]{Book}. %Since $R$ is a two-dimensional normal domain, it follows that $R$ is Cohen-Macaulay. Moreover $R$ has minimal multiplicty, see, for example, \cite[6.36]{Book}.
Thus $R$ is Golod \cite[5.2.8]{Av2}. Hence each $R$-module $M$ with $\pd(M)=\infty$ is a test module over $R$; see (\ref{onGolod}). In particular, if $R$ is not Gorenstein, then the dualizing module of $R$ is a test module.%; see also  Proposition \ref{corgen}.
\end{chunk}

\begin{chunk} \label{Long's result} Let $R=k[\![x_{1}, \ldots, x_{n}]\!]/(f)$, where $k$ is a field and $n\geq 3$. Assume $R$ has an isolated singularity, i.e., $R_{\frak{p}}$ is regular for all $\frak{p} \in \Spec(R)-\{\fm\}$. Let $M$ be an $R$-module. If $n=3$ (i.e., $\dim(R)=2$), or $\dim(M)\leq 1$ (e.g., $M$ has finite length), then $M$ is Tor-rigid; see \cite[2.10, 3.4 and 3.6]{Da1}.
\end {chunk}

A technical but rather important point for us is that a rigid-test module, unlike the residue field $k$, may have arbitrary depth and, even if its depth is zero, it does not have to have finite length in general. Here is such an example:

\begin{chunk} \label{egg2} Let $k$ be a field, $R=k[\![x,y,z]\!]/(xy-z^2)$ and $M=\fm/x\fm$. Then $M$ is a rigid-test module with $\depth(M)=0$ and $\dim(M)=1$; see (\ref{onGolod}) and (\ref{Long's result}).
\end{chunk}

The next result, in view of Dao \cite[2.10]{Da1}, was already known when $f$ is homogeneous. Recently Walker \cite{M.E.Walker} removed the homogeneity assumption; see \cite{MPSW1} and also the paragraph preceding \cite[1.2]{M.E.Walker}.

\begin{chunk} \label{Mark} Let $A=Q/(f)$, where $Q=k[x_{1}, \dots, x_{2s+1}]$, with $s\geq 1$, is a polynomial ring (with standard grading) over a perfect field $k$, and $f$ is a nonconstant polynomial in $Q$. Set $R=A_{\fm}$, where $\mathfrak{m}=(x_{1}, \dots, x_{2s+1})A$ and assume $A_{p}$ is a regular for all $\fp \in \Spec(A)-\{\fm\}$. Then all $R$-modules are Tor-rigid; see \cite[2.10]{Da1} and \cite[1.3]{M.E.Walker}.
\end{chunk}

\section{Remarks on two-dimensional rational singularities} In this section we prove an extension of Corollary \ref{corthmint}, and compare our conclusion to a vanishing result obtained in \cite{CeD2}. Furthermore we improve a result of Iyama and Wemyss \cite{IW} over two-dimensional rational singularities; see (\ref{IWobs}).

\begin{chunk} \label{dis1} Let $R$ be a rational singularity as in Corollary \ref{corthmint}, and let $M$ and $N$ be nonzero $R$-modules. Assume $N$ is Tor-rigid and $\pd(N)=\infty$. Then it follows that $N$ is a rigid-test module; see (\ref{onGolod}) and (\ref{Gsing}). Therefore, if $\Ext^n_{R}(M,N)=0$ for some $n\geq \depth(N)$, then Corollary \ref{corofthm-1} implies that $\pd(M)\leq n-1$.
\end{chunk}

If $R$ is as in Corollary \ref{corthmint}, any power of an integrally closed $\fm$-primary ideal is Tor-rigid, and has infinite projective dimension unless $R$ is regular; see the paragraph following Theorem \ref{thmint} and (\ref{inttest}). Therefore (\ref{dis1}) subsumes Corollary \ref{corthmint}.

The following vanishing result is from \cite{CeD2}:

\begin{chunk} (\cite[4.9]{CeD2}) \label{dis2} Let $R$ be a two-dimensional normal local domain such that the class group of $R$ is torsion. Let $M$ and $N$ be nonzero finitely generated $R$-modules. Assume $\CI(M)<\infty$.  Set $c=\cx(M)$, the complexity of $M$ (e.g., $c$ is the codimension of $R$); see \cite[4.2]{Av2}. If $\Ext^{n}_{R}(M,N)=\dots=\Ext^{n+c-1}_{R}(M,N)=0$ for some $n\geq  3-\depth(M)$, then $\Ext^{i}_{R}(M,N)=0$ for all $i\geq  3-\depth(M)$.
\end{chunk}

We will observe that, in characteristic zero, \cite[4.9]{CeD2}, stated as (\ref{dis2}), is essentially  concerned with a hypersurface ring. In the case where $R$ is a hypersurface ring, the conclusion $\Ext^{i}_{R}(M,N)=0$ for all $i\gg 0$ -- in fact $\pd(M)<\infty$ or $\pd(N)<\infty$ -- can be obtained from Corollary \ref{corofthm-1}.

\begin{chunk} \label{dis3} Assume $R$ is complete and contains an algebraically closed field $k$ of characteristic zero. Assume further that $R$ is as in (\ref{dis2}), or equivalently, $R$ is a rational singularity as in Corollary \ref{corthmint}; see \cite[2.6]{CeD2}. 

Then $R$ is Golod; see (\ref{Gsing}). Since $\CI(M)<\infty$, it follows that either $\pd(M)<\infty$, or $\pd(M)=\infty$ and $R$ is a hypersurface that is quotient of a power series ring over $k$ (rational double point); see \cite[5.3.3(2)]{Av2} and \cite[6.18 and 6.37]{Book}. If $\pd(M)<\infty$, then the vanishing result stated in (\ref{dis2}) is vacuous. So assume $\pd(M)=\infty$ and that $R$ is a rational double point. 

Suppose $\pd(N)=\infty$. Then $N$ is a rigid-test module; see (\ref{onGolod}) and (\ref{Long's result}). Hence, if $\Ext^n_{R}(M,N)=0$ for some $n\geq 2$, then Corollary \ref{corofthm-1} implies that $\pd(M)<\infty$. Thus $\pd(N)<\infty$ and $\Ext^{i}_{R}(M,N)=0$ for all $i\geq  3-\depth(M)$; see also \cite[4.7]{AvBu}.
\end{chunk}

Next our aim is to improve the following result of Iyama and Wemyss:

\begin{chunk} \label{IyWe} (Iyama and Wemyss \cite[3.1(1)]{IW}) Let $R$ be a two-dimensional complete normal local domain such that $k$ is algebraically closed and has characteristic zero. Assume $R$ has a rational singularity. Assume further $R$ is not Gorenstein. If $X$ is a maximal Cohen-Macaulay $R$-module such that $\Ext_R^1(X,R)=\Ext_R^2(X,R)=0$, then $X$ is free.
\end{chunk}

The result of Iyama and Wemyss, stated in (\ref{IyWe}), leads to the important result that the category of maximal Cohen-Macaulay $R$-modules does not contain any $n$-cluster tilting objects for $n\geq 3$; see \cite[4.15(3)]{IW} for details. In (\ref{IWobs}) we improve  (\ref{IyWe}) and show that the vanishing of $\Ext_R^1(X,R)$ is not necessary, i.e., we prove that $X$ is free provided that $\Ext_R^2(X,R)=0$.

The proof of (\ref{IyWe}), given in \cite[3.1(1)]{IW}, relies upon on a deep result of Christensen, Piepmeyer, Struli and Takahashi \cite[4.3]{MPSW1}. As we have already noted that two-dimensional rational singularities are Golod (\ref{Gsing}), we can reprove (\ref{IyWe}) by giving a short and straightforward argument without appealing to \cite{MPSW1}.

\begin{chunk} \label{arasonuc} Assume $R$ is as in (\ref{IyWe}) and let $X$ be a maximal Cohen-Macaulay $R$-module such that $\Ext_R^1(X,R)=\Ext_R^2(X,R)=0$. Iyama and Wemyss establishes in \cite[3.1(1)]{IW} that $\Ext_R^i(X,R)=0$ for all $i\geq 1$. We briefly explain their argument for completeness:

As $\Ext_R^1(X,R)=0$, one has $\Omega X \cong X^{\ast}$ where $X^{\ast}=\Hom(M,R)$; see \cite[2.2]{IW}. Similarly, since $\Ext_R^2(X,R)=\Ext_R^1(\Omega X,R)=\Ext_R^1(X^{\ast},R)=0$, it follows that $\Omega X^{\ast}\cong X^{\ast\ast}$. Notice $X$ is reflexive; see, for example, \cite[1.4.1(b)]{BH}. Thus there are exact sequences $0 \to X^{\ast} \to F \to X \to 0$ and $0 \to X \to G \to X^{\ast} \to 0$, where $F$ and $G$ are free $R$-modules. Consequently, using dimension shifting along those exact sequences, one concludes that $\Ext_R^i(X,R)=0$ for all $i\geq 1$.
\end{chunk}

\begin{chunk} \label{propCM} Let $R$ be a local ring, $M$ a nonzero test module and let $N$ be a nonzero  Cohen-Macaulay $R$-module. Assume $\Ext^{i}_{R}(M,N)=0$ for all $i\gg 0$. Then
$\id(N)<\infty$.
\end{chunk}

\begin{proof} We proceed by induction on $\dim(N)$. It follows from our assumption that $\Tor_i^R(M, \Hom(N,E))=0$ for all $i\gg0$, where $E$ is the injective hull of $k$. Hence the case where $\dim(N)=0$ folllows by Matlis duality. Next assume $\dim(N)\geq 1$, and choose $x\in R$ such that $x$ is a non zero-divisor on $N$.  The short exact sequence
$0\rightarrow N\stackrel{x}{\longrightarrow} N\rightarrow N/xN\rightarrow0$
induces the following long exact sequence: $$\cdots\rightarrow\Ext^i_R(M,N)\stackrel{x}{\longrightarrow}\Ext^i_R(M,N)\rightarrow\Ext^i_R(M,N/xN)
\rightarrow\cdots$$
Therefore $\Ext^i_R(M, N/xN)=0$ for
all $i\gg0$. It now follows from the induction hypothesis that
$\id (N/xN)<\infty$. Hence, applying $\Hom^i_R(k,-)$ to the short exact sequence above, we see that $\Ext^i_R(k,N)=0$ for all $i\gg 0$, i.e., $\id(N)<\infty$.
\end{proof}

\begin{proof}[Alternate proofs of (\ref{IyWe})] 
As $R$ is Golod and not Gorenstein, the vanishing of $\Ext_R^i(X,R)$ for all $i\geq 1$ implies that $\pd(X)<\infty$; see (\ref{arasonuc}) and \cite[1.4]{JorSega}. This forces $X$ to free as it is a maximal Cohen-Macaulay $R$-module.

As an alternate argument, one can use (\ref{propCM}): suppose $X$ is not free. Then, as $R$ is Golod, $X$ is a test module; see (\ref{onGolod}). Moreover $\Ext_R^i(X,R)=0$ for all $i\geq 1$; see (\ref{arasonuc}). It now follows from (\ref{propCM}) that $R$ is Gorenstein. Thus $X$ must be free.
\end{proof}

\begin{chunk}\label{ABresult} (Auslander and Bridger; see \cite[2.17 and 2.21]{AuBr}) Let $R$ be a Noetherian ring and let $X$ be a nonzero finitely generated $R$-module. Then there exists an exact sequence $0 \to F \to Y \to X \to 0$ of finitely generated $R$-modules, where $F$ is free, $Y=\Tr\Omega\Tr\Omega X$ and $\Ext^1_R(Y, R) =0$; see (\ref{AuBrsequence}).
\end{chunk}

\begin{chunk} \label{IWobs} Let $R$ be a two-dimensional rational singularity as in (\ref{IyWe}). Assume $X$ is a maximal Cohen-Macaulay $R$-module such that $\Ext_R^2(X,R)=0$. Then $X$ is free.
\end{chunk}

\begin{proof} There exists an exact sequence
$0 \to F \to Y \to X \to 0$, where $F$ is free and $Y=\Tr\Omega\Tr\Omega X$; see (\ref{ABresult}) and also (\ref{AuBrsequence}). This shows that $Y$ is maximal Cohen-Macaulay (since $X$ is maximal Cohen-Macaulay) and induces the long exact sequence:
$\cdots \to \Ext^2_R(X,R) \to \Ext^2_R(Y,R) \to \Ext^2_R(F,R) \to \cdots$. Hence $\Ext^2_R(Y,R)=0$. Furthermore, by \cite[2.17]{AuBr}, we have that $\Ext^1_R(Y, R) =0$. Therefore $\Ext^1_R(Y, R) =\Ext^2_R(Y,R)=0$. Applying (\ref{IyWe}), we conclude that $Y$ is free. Consequently, since $Y$ and $F$ is free, $X$ is free.
\end{proof}

An immediate consequence of (\ref{IWobs}) is:

\begin{chunk} Let $R$ be a rational singularity as in (\ref{IyWe}). If $X$ is an $R$-module such that $\Ext_R^n(X,R)=0$ for some $n\geq 4$, then $\pd(X)<\infty$.
\end{chunk}

\section*{Acknowledgments}
We are grateful to Luchezar Avramov, David Buchsbaum, Kamran Divaani-Aazar, Henrik Holm, Li Liang, Saeed Nasseh, Greg Piepmeyer, William Sanders, Alex Tchernev, Sean Sather-Wagstaff, Yang Zheng, Ryo Takahashi and Mark E. Walker for valuable information and discussions related to this work during its preparation.

\def\cfudot#1{\ifmmode\setbox7\hbox{$\accent"5E#1$}\else
  \setbox7\hbox{\accent"5E#1}\penalty 10000\relax\fi\raise 1\ht7
  \hbox{\raise.1ex\hbox to 1\wd7{\hss.\hss}}\penalty 10000 \hskip-1\wd7\penalty
  10000\box7}

\end{document}